\documentclass[12pt]{article}

\usepackage[margin=.7in]{geometry}
\usepackage{amsfonts}
\usepackage[mathcal]{euscript}
\usepackage{amsmath}
\usepackage{amsthm}
\usepackage{graphicx}

\usepackage{algorithmicx}
\usepackage{algorithm}
\usepackage[noend]{algpseudocode}

\newcommand{\dx}{\partial_x}
\newcommand{\dy}{\partial_y}
\newcommand{\mfn}{M}
\newcommand{\Lal}{\mathcal L_{\mfn,\alpha}}

\newcommand{\CoPo}{\mathbb A}
\newcommand{\HoPo}[1]{\mathbb A_{#1}}

\newcommand{\xG}{x_c}
\newcommand{\yG}{y_c}
\newcommand{\pG}{(x_c,y_c)}
\newcommand{\kappqAss}{Assume that the parameter $q\in\mathbb N$ is given and that $\kappa^2$ is a function of class $\mathcal C^{q-1}$ in a neighborhood of a given point $\pG\in\mathbb R^2$. }
\newcommand{\kappqAssnz}{Assume that the parameter $q\in\mathbb N$ is given and that $\kappa^2$ is a function of class $\mathcal C^{q-1}$ in a neighborhood of a given point $\pG\in\mathbb R^2$ with $\kappa^2\pG\neq 0$. }
\newcommand{\muAss}{Consider the set of unknowns  $\{\mu_{i_x,i_y}, (i_x,i_y)\in\mathbb N_0^2, i_x+i_y\leq q+1\}$ constructed in Algorithm \ref{Algo:AbGPW}, }

\newtheorem{thm}{Theorem}

\newtheorem{prop}{Proposition}
\newtheorem{dfn}{Definition}
\newtheorem{lmm}{Lemma}

\newtheorem{rmk}{Remark}

\usepackage{algorithmicx}
\usepackage{algorithm}
\usepackage[noend]{algpseudocode}

\usepackage{framed}
\usepackage{xcolor} 
\usepackage{tikz,pgfplots}
\pgfplotsset{compat=newest,
legend style={
font=\footnotesize,
rounded corners=2pt,
}}
\pgfplotscreateplotcyclelist{mycolorlist}{%
blue,every mark/.append style={fill=blue!80!black},mark=*\\%
red,every mark/.append style={fill=red!80!black},mark=square*\\%
brown!60!black,every mark/.append style={fill=brown!80!black},mark=otimes*\\%
black,mark=star\\%
blue,every mark/.append style={fill=blue!80!black},mark=diamond*\\%
red,densely dashed,every mark/.append style={solid,fill=red!80!black},mark=*\\%
brown!60!black,densely dashed,every mark/.append style={
solid,fill=brown!80!black},mark=square*\\%
black,densely dashed,every mark/.append style={solid,fill=gray},mark=otimes*\\%
blue,densely dashed,mark=star,every mark/.append style=solid\\%
red,densely dashed,every mark/.append style={solid,fill=red!80!black},mark=diamond*\\%
}

\pgfplotscreateplotcyclelist{nequalstwenty}{%
blue,every mark/.append style={fill=blue!80!black},mark=*\\%
red,every mark/.append style={fill=red!80!black},mark=square*\\%
brown!60!black,every mark/.append style={fill=brown!80!black},mark=otimes*\\%
black,mark=star\\%
green,every mark/.append style={fill=green!80!green},mark=diamond*\\%
pink,every mark/.append style={solid,fill=pink!80!pink},mark=*\\%
yellow!60!yellow,every mark/.append style={solid,fill=yellow!80!yellow},mark=square*\\%
black,every mark/.append style={solid,fill=gray},mark=otimes*\\%
blue,mark=star,every mark/.append style=solid\\%
red,every mark/.append style={solid,fill=red!80!black},mark=diamond*\\%
blue,every mark/.append style={fill=red!80!red},mark=*\\%
red,every mark/.append style={fill=blue!80!blue},mark=square*\\%
brown!60!black,every mark/.append style={fill=brown!80!black},mark=otimes*\\%
green,every mark/.append style={fill=blue!80!green},mark=diamond*\\%
pink,every mark/.append style={solid,fill=red!80!pink},mark=*\\%
yellow!60!yellow,every mark/.append style={solid,fill=brown!80!black},mark=square*\\%
black,every mark/.append style={solid,fill=red},mark=*\\%
green,every mark/.append style={fill=green!80!green},mark=star\\%
pink,every mark/.append style={solid,fill=pink!80!pink},mark=star\\%
yellow!60!yellow,every mark/.append style={solid,fill=yellow!80!yellow},mark=star\\%
}

\title{
Amplitude-based Generalized Plane Waves: new Quasi-Trefftz functions for scalar equations in 2D
}

\author{Lise-Marie Imbert-G\'erard\footnote{
University of Arizona, {\it lmig@math.arizona.edu}.
\underline{Acknowledgement} The research reported here was supported by the NSF grant DMS-1818747.
}}

\begin{document}
\maketitle

\begin{abstract}
%
Generalized Plane Waves (GPWs) were introduced to take advantage of Trefftz methods for problems modeled by variable coefficient equations. Despite the fact that GPWs do not satisfy the Trefftz property, i.e. they are not exact solutions to the governing equation, they instead satisfy a quasi-Trefftz property: they are only approximate solutions. They lead to high order numerical methods, and this quasi-Trefftz property is critical for their numerical analysis.

The present work introduces a new family of GPWs, amplitude-based. The motivation lies in the poor behavior of the phase-based GPW approximations in the pre-asymptotic regime, which will be tamed by avoiding high degree polynomials within an exponential.
The new ansatz is introduces higher order terms in the amplitude rather than the phase of a plane wave as was initially proposed.
The new functions' construction and the study of their interpolation properties are guided by the roadmap proposed in \cite{IGS}. For the sake of clarity, the first focus is on the two-dimensional Helmholtz equation with spatially-varying wavenumber. The extension to a range of operators allowing for anisotropy in the first and second order terms follows. Numerical simulations illustrate the theoretical study of the new quasi-Trefftz functions.
\end{abstract}

\section{Introduction}
Our interest lies in the numerical simulation of time-harmonic wave propagation in inhomogeneous media, for boundary value problems with a scalar-valued governing equation. While homogeneous media lead to governing Partial Differential Equations (PDEs) with constant coefficients, inhomogeneous media lead to variable-coefficient PDEs. 
In this article we introduce a new family of basis functions to be used as a basis in a quasi-Trefftz numerical method, show how to compute them, and study their interpolation properties.
Trefftz methods, a particular type of Galerkin methods, have the specificity to rely on function spaces of solutions to the governing PDE as opposed to standard function spaces of sufficiently smooth functions. They take advantage of this specificity via the derivation of a weak formulation of the boundary value problem of interest, leading to a weak formulation with no volume term. This of course results in a considerable reduction of the discretization's computational cost. We will hereafter refer to solutions of the governing PDE as Trefftz functions, and refer to spaces of Trefftz functions as Trefftz spaces.

The initial idea of using solutions to the governing PDE was introduced by Trefftz in 1926 \cite{Trefftz}, more recently translated from German \cite{transl}, to obtain estimates on the solution of boundary value problems. It was then developed as early as in the 30s under the name of Trefftz method to study for instance problems of elasticity \cite{Elast47,Elast53} or torsion \cite{Torsion31,Torsion48}.
Later the method was fruitfully extended towards numerical approximation. In \cite{Collatz}, the author classifies the method as a boundary method, as the solution to a boundary value problem is approximated by a linear combination of Trefftz functions while this linear combination is constructed to take into account the boundary condition, see also \cite{Herrera}.  Trefftz functions were used as a local basis on a region and coupled to a Finite Element region in  \cite{Zienk}. In \cite{Jirous77},  they were used as a local basis on individual mesh elements instead of polynomial basis functions; an overview of such early Trefftz-type Galerkin methods can be found in \cite{Jirous}.

Several versions of Trefftz methods have been actively developed towards the numerical simulation of time-harmonic wave propagation problems for more than twenty years \cite{survey}, such as the Trefftz-Discontinuous-Galerkin method \cite{git,GH14} or the Ultra-Weak Variational Formulation \cite{uwvf,cess,cd98}. More recent development include the Trefftz Virtual Element Method \cite{Tvem1,Tvem2}. Trefftz spaces as well as Trefftz functions are absolutely fundamental to all Trefftz methods, throughout the weak formulation derivation as well as throughout the analysis of the resulting numerical method. 

Ideally the discretization of a Trefftz weak formulation is performed via a finite-dimensional vector subspace of the Trefftz space. 
Trefftz functions are available for several problems of time-harmonic wave propagation through homogeneous media, i.e. when the governing PDE has constant coefficients, for instance the most obvious plane, circular or spherical waves, but also functions constructed from Bessel functions \cite{Zienk91,Teemu}. By contrast, for wave propagation through inhomogeneous media, i.e. when the governing PDE has either smooth or piecewise smooth variable coefficients, in general Trefftz functions are not available. 
Nevertheless the idea of 
   replacing Trefftz functions by  approximate solutions
to the governing PDE to discretize a Trefftz formulation when the PDE has variable coefficients was introduced in \cite{LMD}, and various aspects of these functions, called Generalized Plane Waves (GPWs), as well as the resulting method, were studied in \cite{LMinterp,MC,IGM,IGS}.  

In general we will refer to approximate solutions to the PDE as quasi-Trefftz functions, and to associated methods as quasi-Trefftz methods. 
As the PDE coefficients are variable, there is no hope to guarantee global approximation properties, and therefore the construction of quasi-Trefftz functions should emphasize local properties.
For instance GPWs are constructed locally and satisfy a local approximation of the PDE in the sense of a Taylor expansion. In general, quasi-Trefftz functions are defined piece-wise, element per element on a mesh of the computational domain, and the approximation of the PDE is expected to be accurate locally on each element. We will therefore focus on the neighborhood of a generic point $(x_c,y_c)\in\mathbb R^2$, that could for instance be the centroid of each mesh element. 

The GPWs introduced in \cite{LMD} are phase-based, as they rely on the addition of higher order terms in the phase of classical PWs. Because of the presence of high degree polynomials within an exponential, they exhibited a poor behavior in the pre-asymptotic regime, i.e. at a distance $h\approx 1$ from $(x_c,y_c)$. This motivates the introduction of a new family of GPWs, avoiding the presence of high degree polynomials within an exponential. In this article we introduce a family of amplitude-based GPWs, with a phase identical to that of a PW but with higher degree polynomials in the amplitude.
The first challenge is to design of a new ansatz. 
Thanks to an appropriate choice of ansatz, the problem of constructing a GPW will be reformulated into a problem that has solutions, but most importantly a problem for which a particular solution can be constructed at a limited computational cost.
The second challenge is to decipher the relation between the new functions and PWs in order to study the local interpolation properties of these new GPWs. 


The construction of amplitude-based GPWs and the study of their interpolation properties will respectively be addressed in Sections \ref{sec:constr} and \ref{sec:int}, for the two-dimensional Helmholtz equation:
\begin{equation}
\label{eq:Helm}
\Big[-\Delta-\kappa^2(x,y) \Big] u = 0,
\end{equation}
following the roadmap proposed in \cite{IGS}.
The extension of these results to a set of anisotropic second order PDEs will be the focus of Section \ref{sec:Ext}. Section \ref{sec:NR} presents numerical experiments illustrating interpolation properties and emphasizing the improvement obtained with respect to phase-based GPWs in terms of interpolation properties.

To describe various useful  sets of indices we will write $\mathbb N_0=\mathbb N\cup \{0\}$.

\section{Construction of the Amplitude-based GPWs}
\label{sec:constr}
The construction of this new family of GPWs relies on the choice of a new ansatz, shifting the focus from the phase to the amplitude of PWs. Here we emphasize the motivation behind the new ansatz, and the reformulation of the problem of defining an amplitude-based GPW into a well-posed problem. The final goal of this section is to provide an algorithm to construct such GPWs guaranteeing that they satisfy an approximation of the PDE \eqref{eq:Helm}, see Algorithm \ref{Algo:AbGPW} and Proposition \ref{prop:SatisTE}.

\subsection{Amplitude-based versus phase-based GPWs}
\label{AbvsPb}
The original idea behind GPW was to start by considering a cPW, and its relation to the constant coefficient Helmholtz equation:
$$
\left\{
\begin{array}{l}
W(x,y) = \exp\Big( \lambda_{10} (x-\xG) + \lambda_{01} (y-\yG) \Big)  \quad \text{where}  \quad ( \lambda_{10}, \lambda_{01})\in\mathbb C^2,\\
(-\Delta-\kappa^2 )W = 0 \Leftrightarrow \lambda_{10}^2+ \lambda_{01}^2 +\kappa^2 = 0.
\end{array}
\right.
$$
In this case the ansatz chosen for $W$ has two degrees of freedom, namely $( \lambda_{10}, \lambda_{01})\in\mathbb C^2$. The single constraint $\lambda_{10}^2+ \lambda_{01}^2 +\kappa^2 = 0$ on these two degrees of freedom is sufficient to ensure that the ansatz $W$ solves the governing PDE. Moreover, by choosing $( \lambda_{10}, \lambda_{01}) = \mathrm i \kappa (\cos\theta,\sin\theta)$ for some real parameter $\theta$, the constraint is satisfied and any family of such functions, associated to any distinct values of $\theta\in[0,2\pi)$, is linearly independent so it is a basis for a finite-dimensional subspace of the Trefftz space for the Helmholtz equation. This subspace has been the most widely used to discretize Trefftz weak formulations.

In the case of a variable coefficient $\kappa(x,y)$ however, there is no general closed formula for exact solutions of the PDE. The original idea behind GPW might be summarized in two points:
\begin{itemize}
\item relaxing the Trefftz property, $(-\Delta-\kappa^2 )W = 0$, into an approximation $(-\Delta-\kappa^2 )G \approx 0$;
%
\item choosing an ansatz with more degrees of freedom by adding for higher order terms (HOT) to the phase of the classical PW: $G(x,y) = \exp\Big( \lambda_{10} (x-\xG) + \lambda_{01} (y-\yG) +HOT\Big)$.
\end{itemize}
The image of $G$ by the differential operator, i.e. the function $(-\Delta-\kappa^2 )G$, won't be zero, but instead it will locally approximate zero: it is the Taylor polynomial of this function that will be equal to zero. So the parameter $q$ will refer to the order of the Taylor expansion approximation. Throughout the construction process, $q$'s value remains unconstrained, it only comes into play to guarantee high order interpolation properties. A GPW was then initially defined in the vicinity of a point $\pG\in\mathbb R^2$ as
$$
\left\{
\begin{array}{l}
G(x,y) := \exp P(x,y) \text{ with } P\in\mathbb C[X,Y] \text{ such that }\\
(-\Delta-\kappa^2(x,y) )G(x,y) = O(|(x,y)-\pG|^q).
\end{array}
\right.
$$
As a consequence, the construction of a GPW was equivalent to the following problem
$$
\left\{
\begin{array}{l}
\text{Find a polynomial } P\in\mathbb C[X,Y] \text{ such that }\\
G(x,y) := \exp P(x,y) \text{ satisfies }  
(-\Delta-\kappa^2(x,y) )G(x,y) = O(|(x,y)-\pG|^q)
\end{array}
\right.
$$
Moreover, since
$$
(-\Delta-\kappa^2(x,y) )\exp P(x,y) =\big(-\Delta P (x,y) - |\nabla P(x,y)|^2-\kappa^2(x,y) \big)\exp P(x,y)
$$
while $\exp P$ is bounded in the vicinity of $\pG$,  the construction of a GPW was also equivalent to the following problem
\begin{equation}
\label{pb:PbGPW}
\left\{
\begin{array}{l}
\text{Find a polynomial } P\in\mathbb C[X,Y] \text{ such that }\\
-\Delta P (x,y) - |\nabla P(x,y)|^2-\kappa^2(x,y) = O(|(x,y)-\pG|^q)\\
G(x,y) := \exp P(x,y)
\end{array}
\right.
\end{equation}

We are now interested in exploring a new type of quasi-Trefftz functions, this time choosing an ansatz  with more degrees of freedom as higher order terms (HOT) added to the amplitude rather than the phase of a classical PW: $G(x,y) = (1+ HOT)\exp\Big( \lambda_{10} (x-\xG) + \lambda_{01} (y-\yG) \Big)$. An amplitude-based GPW is then  defined in the vicinity of a point $\pG\in\mathbb R^2$ as
$$
\left\{
\begin{array}{l}
G(x,y) := Q(x,y)\exp \Big( \lambda_{10} (x-\xG) + \lambda_{01} (y-\yG) \Big) \text{ with } Q\in\mathbb C[X,Y], ( \lambda_{10}, \lambda_{01})\in\mathbb C^2 \text{ such that }\\
(-\Delta-\kappa^2(x,y) )G(x,y) = O(|(x,y)-\pG|^q).
\end{array}
\right.
$$
As a consequence, the construction of an amplitude-based GPW is equivalent to the following problem
$$
\left\{
\begin{array}{l}
\text{Find } (Q, ( \lambda_{10}, \lambda_{01}) )\in\mathbb C[X,Y]\times\mathbb C^2 \text{ such that }\\
G(x,y) := Q(x,y)\exp \Big( \lambda_{10} (x-\xG) + \lambda_{01} (y-\yG) \Big) \text{ satisfies }  \\
(-\Delta-\kappa^2(x,y) )G(x,y) = O(|(x,y)-\pG|^q)
\end{array}
\right.
$$
Moreover, defining $\vec{d}:= ( \lambda_{10}, \lambda_{01})$ and $\widetilde \kappa^2:= \left|\vec{d}\right|^2+\kappa^2$, since
$$
\begin{array}{l}
(-\Delta-\kappa^2(x,y) )Q(x,y)\exp ( \lambda_{10} (x-\xG) + \lambda_{01} (y-\yG)) 
\\
=\left(-\Delta Q (x,y) -2 \nabla Q(x,y)\cdot \vec{d}-\widetilde \kappa^2(x,y)Q(x,y) \right)\exp \Big( \lambda_{10} (x-\xG) + \lambda_{01} (y-\yG) \Big) 
\end{array}
$$
while $\exp \Big( \lambda_{10} (x-\xG) + \lambda_{01} (y-\yG) \Big) $ is bounded in the vicinity of $\pG$,  the construction of an amplitude-based GPW is also equivalent to the following problem 
\begin{equation}
\label{pb:AbGPW}
\left\{
\begin{array}{l}
\text{Find } \left(Q, \vec{d}\right)\in\mathbb C[X,Y]\times\mathbb C^2 \text{ such that }\\
-\Delta Q (x,y) -2 \nabla Q(x,y)\cdot \vec{d}-\widetilde \kappa^2(x,y)Q(x,y) = O(|(x,y)-\pG|^q)\\
G(x,y) := Q(x,y)\exp 
\vec{d}\cdot \begin{pmatrix} x-\xG\\y-\yG\end{pmatrix}
\end{array}
\right.
\end{equation}
\begin{rmk}
Comparing \eqref{pb:AbGPW} to \eqref{pb:PbGPW} we observe that the new problem is not linear anymore.
\end{rmk}
However, there is no guarantee that this problem is well-posed, hence Subsections \ref{ssec:Lsys} and \ref{ssec:Layer} turn to the study of well-posedness of this problem.
Next, given both the order $q$ and the center $\pG$ of the Taylor expansion, we will focus on developing an algorithm to construct such a GPW, which is equivalent to computing the coefficients of the polynomial $Q$ so that the coefficients of the Taylor expansion of $-\Delta Q -2 \nabla Q\cdot \vec{d}-\widetilde \kappa^2Q$ are zero, see Subsection \ref{ssec:const}. Beyond the construction of one GPW, the goal is evidently to construct a family of linearly independent GPWs with good interpolation properties, and these aspects will be addressed in Section \ref{sec:int}.



\subsection{A linear system}\label{ssec:Lsys}
In order to underline the structure of the problem at stake, we will now consider  $\vec{d}:= ( \lambda_{10}, \lambda_{01})\in\mathbb C^2$ to be fixed while we identify any polynomial $Q\in\mathbb C[X,Y]$ to the set of its coefficients $\{\mu_{i_x,i_y}, (i_x,i_y)\in\mathbb N_0^2, i_x+i_y\leq \deg Q\}$ via $\displaystyle Q(x,y)=\sum_{0\leq i_x+i_y\leq \deg Q} \mu_{i_x,i_y} (x-\xG)^{i_x} (y-\yG)^{i_y}$. Later $\vec{d}$ will be used to define a family of GPWs. In this case, \eqref{pb:AbGPW} boils down to a linear system:
\begin{itemize}
\item the unknowns are $\{\mu_{i_x,i_y}, (i_x,i_y)\in\mathbb N_0^2, i_x+i_y\leq \deg Q\}$, 
\item the equations are $\left\{\dx^{j_x}\dy^{j_y}\big[ -\Delta Q -2 \nabla Q\cdot \vec{d}-\widetilde \kappa^2Q \big]\pG = 0, (j_x,j_y)\in\mathbb N_0^2, j_x+j_y\leq q-1\right\}$.
\end{itemize}
This linear system therefore has $N_{dof}:=\frac{(\deg Q +1)(\deg Q +2)}{2}$ unknowns and $N_{eqn}:=\frac{q(q +1)}{2}$ equations, and in order to build a solution to Problem \eqref{pb:AbGPW}. For clarity we will consistently use indices $(i_x,i_y)$ to refer to unknowns while indices $(j_x,j_y)$ will refer to equations.

Obviously, choosing the degree of $Q$, $\deg Q$, will affect the well-posedness of this system, as depending on $\deg Q+1$ being smaller than, larger than or equal to $q$ the system is respectively overdetermined, underdetermined or square. 
Similarly to the choice of $\deg P$ for phase-based GPWs, we will choose $\deg Q = q+1$ to exploit the structure provided by the Laplacian, see Remark \ref{rmk:notosneat}.
According to this choice, as $\dx^{j_x}\dy^{j_y} Q \pG  = j_x!j_y!\mu_{j_x,j_y}$ for $j_x+j_y\leq q+1$ and zero otherwise, the linear system can then be written
\begin{equation}
\label{LinSys}
\left\{
\begin{array}{l}
\forall (j_x,j_y)\in\mathbb N_0^2, j_x+j_y\leq q-1,\\
(j_x+2)(j_x+1)\mu_{j_x+2,j_y} + (j_y+2)(j_y+1)\mu_{j_x,j_y+2}
+2(j_x+1)\lambda_{10}\mu_{j_x+1,j_y}+2(j_y+1)\lambda_{01}\mu_{j_x,j_y+1}
\\\displaystyle \phantom{(j_x+2)}
+\sum_{k_x=0}^{j_x}\sum_{k_y=0}^{j_y} \frac{1}{(j_x-k_x)!(j_y-k_y)!}\dx^{j_x-k_x}\dy^{j_y-k_y}\widetilde \kappa^2 \pG \mu_{k_x,k_y} = 0
\end{array}
\right.
\end{equation}
The goal in this section is then to leverage the structure of this linear system in order to build non trivial solutions and hence build corresponding GPWs $\displaystyle Q(x,y)=\sum_{0\leq i_x+i_y\leq q+1} \mu_{i_x,i_y} (x-\xG)^{i_x} (y-\yG)^{i_y}$.

\begin{rmk}
\label{rmk:notosneat}
Note that for any value of  $\deg Q\geq q+1$ the corresponding linear system could be written exactly as \eqref{LinSys}. Therefore, picking $\deg Q>q+1$ would would increase the number of degrees of freedom without affecting the properties of the system. 

On the other hand, if we had chosen $\deg Q<q+1$, at least the equations $(j_x,j_y)$ such that $j_x+j_y=\deg Q-1$ would read
\begin{equation}
\left\{
\begin{array}{l}
\forall (j_x,j_y)\in\mathbb N_0^2, j_x+j_y=\deg Q-1,\\
2(j_x+1)\lambda_{10}\mu_{j_x+1,j_y}+2(j_y+1)\lambda_{01}\mu_{j_x,j_y+1}
\\\displaystyle \phantom{(j_x+2)}
+\sum_{k_x=0}^{j_x}\sum_{k_y=0}^{j_y} \frac{1}{(j_x-k_x)!(j_y-k_y)!}\dx^{j_x-k_x}\dy^{j_y-k_y}\widetilde \kappa^2 \pG \mu_{k_x,k_y} = 0
\end{array}
\right.
\end{equation}
and the upcoming study of the system would not hold.
\end{rmk}

\subsection{Layer structure and well-posedness}\label{ssec:Layer}
Even though there are no non linear terms in the system for amplitude-based GPWs, the structure of the system is analogous to that of the equivalent system for phase-based GPWs. We can gather unknowns $\mu_{i_x,i_y}$ according to the total degree of their monomial in $Q$, that is to say according to $i_x+i_y$.  Likewise, for each value of $\ell$ from $0$ to $q-1$ we can gather equations according to $\ell=j_x+j_y$ into subsystems of $\ell+1$ equations for the unknowns $\{\mu_{i_x,i_y}, (i_x,i_y)\in\mathbb N_0^2, i_x+i_y=\ell+2\}$. Therefore each subsystem can be rewritten as
\begin{equation}
\label{LinSubSys}
\left\{
\begin{array}{l}
\forall (j_x,j_y)\in\mathbb N_0^2, j_x+j_y=\ell,\\
(j_x+2)(j_x+1)\mu_{j_x+2,j_y} + (j_y+2)(j_y+1)\mu_{j_x,j_y+2}\\
=-2(j_x+1)\lambda_{10}\mu_{j_x+1,j_y}-2(j_y+1)\lambda_{01}\mu_{j_x,j_y+1}
\\\displaystyle \phantom{(j_x+2)}
-\sum_{k_x=0}^{j_x}\sum_{k_y=0}^{j_y} \frac{1}{(j_x-k_x)!(j_y-k_y)!}\dx^{j_x-k_x}\dy^{j_y-k_y}\widetilde \kappa^2 \pG \mu_{k_x,k_y} .
\end{array}
\right.
\end{equation}
The right hand side of each equation in \eqref{LinSubSys} only involves unknowns $\{\mu_{i_x,i_y}, (i_x,i_y)\in\mathbb N_0^2, i_x+i_y<\ell+2\}$. The subsystems can then be considered sequentially for increasing values of $\ell$, as the right hand side would then be know from the previous layers $\widetilde \ell <\ell$. Each subsystem consists of $\ell+1$ equations and has $\ell+3$ unknowns, namely $\{\mu_{i_x,i_y}, (i_x,i_y)\in\mathbb N_0^2, i_x+i_y=\ell+2\}$, furthermore the three unknowns $\{\mu_{0,0},\mu_{1,0},\mu_{0,1}  \}$ only appear in the right hand sides of these subsystems.

We can now justify why each subsystem has a solution, thanks to a reformulation of each subsystem in terms of the partial operator operator $\Delta$ defined on the space of homogeneous polynomials.
%
Denote by $\CoPo= \mathbb C[X,Y]$ the space of complex polynomials in two variables, and by $\HoPo{d}\subset \CoPo$ the space of  homogeneous polynomials of degree $d$. Since $\HoPo{d} = Span \{ X^iY^{d-i}, 0\leq i\leq d \}$, it is clear that $\dim \HoPo{d} = d+1$.
For a given level $\ell\in\mathbb N_0$, consider the Laplacian of a homogeneous polynomial of degree $\ell+2$:
\begin{equation}
\label{eq:LpHoPo}
\Delta \left[\sum_{i=0}^{\ell+2} p_i X^iY^{\ell+2-i}\right]
=
\sum_{i=0}^{\ell} \Big((i+2)(i+1) p_{i+2} + (\ell+2-i)(\ell+1-i)p_i\Big) X^iY^{\ell-i}.
\end{equation}
Then the restriction $\Delta_\ell$ of the Laplacian operator to a space of homogeneous polynomials $\HoPo{\ell+2}$ is defined as:
$$
\begin{array}{rccc}
\Delta_\ell: &\HoPo{\ell+2}&\rightarrow &\HoPo{\ell}\\
&P &\mapsto& \Delta P
\end{array}
$$
and we are interested in the range of this linear operator. Indeed, the Subsystem \eqref{LinSubSys} has a solution if and only if the operator $\Delta_\ell$ is surjective.

Let's focus first on the kernel of $\Delta_\ell$. It is clear from \eqref{eq:LpHoPo} that 
$$
\begin{array}{rl}
\ker \Delta_\ell = & \displaystyle
\Bigg\{
\sum_{i=0}^{\ell+2} p_i X^iY^{\ell+2-i}, \{p_i\}_{0\leq i \leq \ell+2}\in\mathbb C^{\ell+3}, \\
&
\phantom{ \Bigg\{ }
(i+2)(i+1) p_{i+2} + (\ell+2-i)(\ell+1-i)p_i = 0
\text{ for } 0\leq i \leq \ell
\Bigg\}
,
\end{array}
$$
which is equivalent to
$$
\begin{array}{rl}
\ker \Delta_\ell = & \displaystyle
Span\Bigg\{
 \sum_{i=0}^{\left\lfloor \frac{\ell+2}{2}\right\rfloor} 
 \frac{(\ell+2)!}{(2i)!(\ell+2-2i)!}
X^{2i}Y^{\ell+2-2i},\\
& \displaystyle
\phantom{ Span\Bigg\{ }
\sum_{i=1}^{\left\lfloor \frac{\ell+3}{2}\right\rfloor} 
 \frac{(\ell+1)!}{(2i-1)!(\ell+3-2i)!}
X^{2i-1}Y^{\ell+3-2i}
\Bigg\}
.
\end{array}
$$
So $\dim( \ker\Delta_\ell )= 2$.
As a consequence, since  $\dim \HoPo{\ell+2} = \ell+3$ while $\dim \HoPo{\ell} = \ell+1$, the rank-nullity theorem shows that the operator $\Delta_\ell$ is full-rank.

This then shows that each subsystem \eqref{LinSubSys} has a solution. 

In turn, thanks to the layer structure of System \eqref{LinSys}, we have then proved that the system has a solution.
With $N_{dof}=\frac{(q+2)(q+3)}{2}$ unknowns and  $N_{eqn}=\frac{q(q+1)}{2}$ equations, System \eqref{LinSys} can be solved using $N_{dof}-N_{eqn}=2q+3$ appropriate additional constraints. First, the three unknowns $\{\mu_{0,0},\mu_{1,0},\mu_{0,1}  \}$ can be fixed, as we have already noted that they appear only on right hand sides of the subsystems. 
Next, for increasing values of $\ell$, each of the $q$ subsystem is triangular. This can be evidenced via numbering of the equations with increasing values of $j_x$ and numbering of the unknowns with increasing values of $i_x$. It is then clear that by adding the two additional constraints corresponding to fixing $\mu_{0,\ell +2}$ and $\mu_{1,\ell+1}$ we obtain a well-posed problem for each layer $\ell$. The three initial constraints plus the two constraints per subsystem altogether form $2q+3$ additional constraints, and there is a unique solution to \eqref{LinSys} augmented by these constraints. 

  In summary, the problem of constructing a GPW can then be reformulated into a well-posed problem:
  \begin{equation}
  \label{pb:reform}
  \left\{
  \begin{array}{l}
  \text{Fix }  \{\mu_{0,0},\mu_{1,0},\mu_{0,1}  \}, \\
  \begin{array}{rl}
    \forall \ell \in [\!\![ 0,q-1]\!\!]&
   \text{Fix }  \mu_{0,\ell +2},\mu_{1,\ell+1} ,\\
   &\text{Solve \eqref{LinSubSys}}.
  \end{array}
  \end{array}
  \right.
  \end{equation}

\subsection{Construction of solutions}\label{ssec:const}
The well-posedness of \eqref{pb:reform} -- and therefore the existence of a solution to \eqref{pb:AbGPW} -- is independent of the values chosen to fix the additional constraints. However, in order for the resulting GPW to have useful interpolation properties, we make the following choices.
To obtain a GPW of the form $G(x,y) = (1+ HOT)\exp\Big( \lambda_{10} (x-\xG) + \lambda_{01} (y-\yG) \Big)$ as announced, we choose to fix $\mu_{0,0}=1$, as any choice of a non-zero value independent of  $( \lambda_{10},\lambda_{01})$ would not affect any of the results that follow. Next we choose for simplicity to fix $\mu_{1,0}=0$ and $\mu_{0,1}=0$, and,  for increasing values of $\ell$,  $\mu_{0,\ell+2}=0$ and $\mu_{\ell+2,0}=0$.
The unique solution to the reformulated problem \eqref{pb:reform} can be constructed thanks to the following algorithm.
 
\begin{algorithm}[H] 
\caption{Construction of an amplitude-based GPW for the Helmholtz equation}
\label{Algo:AbGPW}
\begin{algorithmic}[1]
\State Given $\vec{d}=(\lambda_{1,0},\lambda_{0,1})\in\mathbb C^2$
and
$\pG\in\mathbb R^2$
\State 
\label{Base}
Fix $\mu_{0,0}=1$ as well as $(\mu_{1,0},\mu_{0,1})=(0,0)$
\For{$\ell \gets 0,\, q-1$}
\State Fix $\mu_{0,\ell +2}=0$ and $\mu_{1,\ell+1}=0$
\For{$j_x \gets 0,\, \ell$}
\State \label{eq:RHS}
$\mathsf{RHS}:=-2(j_x+1)\lambda_{10}\mu_{j_x+1,\ell-j_x} - 2 (\ell-j_x+1) \lambda_{01} \mu_{j_x,\ell-j_x+1}$
\State $\displaystyle \qquad \qquad -\sum_{k_x=0}^{j_x}\sum_{k_y=0}^{\ell-j_x} \frac{1}{(j_x-k_x)!(\ell-j_x-k_y)!}\dx^{j_x-k_x}\dy^{\ell-j_x-k_y}\widetilde \kappa^2 \pG \mu_{k_x,k_y}$
\State \label{inductionform}
$\displaystyle \mu_{j_x+2,\ell-j_x} := \frac{1}{(j_x+2)(j_x+1)}\Big(\mathsf {RHS}- (\ell-j_x+2)(\ell-j_x+1) \mu_{j_x,\ell-j_x+2}\Big)$ 
\EndFor
\EndFor
\State $\displaystyle Q(x,y) \gets \sum_{0\leq i_x+i_y\leq q+1} \mu_{i_x,i_y}(x-\xG)^{i_x}(y-\yG)^{i_y}$
\State $\displaystyle G_{\overrightarrow d}(x,y) \gets Q(x,y)\exp \Big( \lambda_{10} (x-\xG) + \lambda_{01} (y-\yG) \Big)$
\end{algorithmic}
\end{algorithm}
Thanks to the triangular structure of subsystems \eqref{LinSubSys}, which hence are solved by substitution, 
the computational cost of Algorithm \ref{Algo:AbGPW} is simply linear with respect to the number of polynomial coefficients of $Q$. Moreover, even though it does not affect the linear rate, the choice to fix to zero $\mu_{1,0}$, $\mu_{0,1}$, $\mu_{0,\ell+2}$ and $\mu_{\ell+2,0}$ further reduces the computational cost.

From the derivation of this algorithm we immediately obtain the following result.
\begin{prop}
\label{prop:SatisTE}
\kappqAss
 An amplitude-based GPW $G_{\vec d}$ constructed from Algorithm \ref{Algo:AbGPW} for any $\vec{d}=(\lambda_{1,0},\lambda_{0,1})\in\mathbb C^2$ satisfies the local approximation property in the vicinity of $\pG$
\begin{equation}
\label{eq:TaylApp}
(-\Delta-\kappa^2(x,y) )G_{\vec d}(x,y) = O(|(x,y)-\pG|^q).
\end{equation}
\end{prop}
This property actually holds not only independently of the choice of $\vec{d}=(\lambda_{1,0},\lambda_{0,1})\in\mathbb C^2$, but also independently of the other values fixed in the Algorithm, since \eqref{eq:TaylApp} holds simply by construction.
 The particular choice of $\vec d$ that will now be proposed is however the corner stone of the proofs leading to interpolation properties of the new GPWs.

The choice of the $\vec d$, related to the direction of propagation of PWs, is referred to as normalization. In order to build a family of amplitude-based GPWs, we simply mimic classical PW by choosing 
$(\lambda_{1,0},\lambda_{0,1})= \sqrt{-\widetilde\kappa^2\pG} (\cos \theta,\sin\theta)$, and construct the corresponding GPW for any value of $\theta$.

\begin{dfn}
\label{def:GPWsp}
\kappqAssnz
For any set of $p$ angles $\mathrm A:=\{\theta_k \in [0,2\pi); 1\leq k \leq p\}$, 
for each angle we define the associated $\vec{d}_k:= \sqrt{-\kappa^2\pG} (\cos \theta_k,\sin\theta_k)$ and GPW $G_k :=G_{\vec{d}_k}$. We will denote the corresponding set of GPW functions $\mathcal B_{\mathrm A,q}:=\{G_k;\theta_k\in\mathrm A, 1\leq k\leq p\} $.
\end{dfn}

As a byproduct of the proof of interpolation properties, we will prove the linear independence of this set of GPWs under the mere condition that $\mathrm A$ is a set of distinct angles.

Obviously, if $\kappa^2\pG=0$, then all $G_k$ are identical, because then all $\vec{d}_k$ are the same. However, in this case, the normalization could be chosen differently, and a family of linearly independent GPWs would still be constructed. We will not focus on this aspect here, but a study of interpolation property with a normalization independent of $\kappa$ was presented in \cite{IGS} for phase-based GPWs, and the same approach is applicable to the amplitude-based GPWs. Interpolation properties of these new functions indeed hold even with a $\kappa$-independent normalization, so the new GPWs could be used (like the phase-based ones) in domains including points at which $\kappa$ vanishes.
For instance for applications to plasma physics where a cut-off can be defined as a line along which $\kappa$ vanishes, it would be possible to change the normalization in a neighborhood of a cut-off.

\section{Interpolation properties}
\label{sec:int}
We will now follow the roadmap proposed in \cite{IGS} to study interpolation properties of GPWs, adapting each step to this new framework of amplitude-based GPWs. At each step before the last one, we will emphasize interpretation of the central idea in this new framework and state the desired properties in Lemmas, finally resulting in the interpolation properties summarized in Theorem \ref{thm:AbInterp}. The theorem's proof relies on constructing a GPW approximation to any smooth solution of the PDE by matching their Taylor expansions. A natural element coming into play when matching the Taylor expansion of a linear combination of $p$ functions $\{f_i,i\in\mathbb N, i\leq p\}$ is the matrix $\mathsf M_n\in\mathbb C^{(n+1)(n+2)/2\times p}$ built column-wise from the Taylor expansion coefficients of each function, and we will use the notation
\begin{equation}
\label{eq:DefMat}
\forall (j_x,j_y)\in(\mathbb N_0)^2, j_x+j_y\leq n,\quad
(\mathsf M_n)_{\frac{(j_x+j_y)(j_x+j_y+1)}{2}+j_y+1,k} := \dx^{j_x} \dy^{j_y} f_k \pG/(j_x!j_y!).
\end{equation}

\subsection*{ Step 1} 
In this first step, we seek common properties of the unknowns that are computed inside the nested loop in Algorithm \ref{Algo:AbGPW}. In the phase-based context, this amounts to studying the coefficients of highest degree terms in the polynomial $P$ with respect to the only non-zero fixed unknowns in the construction algorithm, namely $(\lambda_{1,0},\lambda_{0,1})$. So it involves exclusively on the phase polynomial. By contrast, here, it will couple the amplitude polynomial with the phase term. The coefficients of highest degree terms in the polynomial $Q$ can't possibly be expressed exclusively in terms of the only non-zero fixed unknown in the construction algorithm, which would be $\mu_{0,0}$. Instead, as appears clearly from the computation of $\mathsf{ RHS}$ in the algorithm, they are intrinsically coupled to the fixed terms from the phase, namely $(\lambda_{1,0},\lambda_{0,1})$. Despite this practical difference, the fundamental idea remains the same independently of the choice of ansatz for the GPW: both ansatz are designed starting from a classical plane wave, so the new unknowns introduced in each ansatz are studied with respect to the parameters defining a classical plane wave, namely $(\lambda_{1,0},\lambda_{0,1})$.

The relation between $\{ \mu_{i_x,i_y}, i_x\geq 2 \}$ and $(\lambda_{0,1},\lambda_{1,0})$ is polynomial, and each of these $\mu_{i_x,i_y}$ has a particular degree as a polynomial in $\mathbb C[\lambda_{1,0},\lambda_{0,1}]$.

\begin{lmm}
\label{lmm:mus}
\kappqAssnz
\muAss
 under the assumption (inspired by classical Plane Waves) that the quantity $ \left( \lambda_{1,0} \right)^2 + \left( \lambda_{0,1} \right)^2$ is 
 equal to $-\kappa^2\pG$.
While $\mu_{0,0}=1$, each $\mu_{i_x,i_y}$ for $i_x+i_y\geq 1$ can be expressed as a polynomial in $\mathbb C[\lambda_{1,0},\lambda_{0,1}]$ of degree at most equal to $i_x+i_y-2$.
\end{lmm}
In order for the following proof to hold, it is sufficient for the quantity $ \left( \lambda_{1,0} \right)^2 + \left( \lambda_{0,1} \right)^2$ to have a fixed value in $v\in\mathbb C$. In other words, as explained in \cite{LMinterp}, instead of considering elements of the polynomial ring $\mathbb C[\lambda_{1,0},\lambda_{0,1}]$, we instead consider polynomials of the quotient ring $\mathbb C[\lambda_{1,0},\lambda_{0,1}]/\left(\left( \lambda_{1,0} \right)^2 + \left( \lambda_{0,1} \right)^2-v\right)$. This explains the phrasing of the Lemma: $\mu_{i_x,i_y}$  {\it can be expressed as} a polynomial with a certain degree -- as opposed to {\it has} a certain degree.

\begin{proof}
In view of Formula (\ref{Base}:) from Algorithm  \ref{Algo:AbGPW}, the result is clear for $(\mu_{1,0},\mu_{0,1})$.
In view of Formula (\ref{inductionform}:) from Algorithm \ref{Algo:AbGPW}, we will further proceed by nested induction on $\ell$ and $j_x$.  Following the layer structure of our problem, we start by the induction with respect to $\ell$.

For $\ell=0$, $\mu_{0,2}$ and $\mu_{1,1}$ are both fixed to zero so they clearly are polynomials in $\mathbb C[\lambda_{1,0},\lambda_{0,1}]$ of degree at most equal to $0$. Then since $\mu_{0,0}=1$ the last $\mu_{i_x,i_y}$ with $i_x+i_y=2$ can be written as 
\begin{equation}
\label{eq:mu20}
\mu_{2,0} 
=\frac12\left( -2\lambda_{1,0}\mu_{1,0} -2\lambda_{0,1}\mu_{0,1} - \widetilde \kappa^2\pG 
\right)
\end{equation}
From the definition of $\widetilde\kappa^2$ and the assumption  on $(\lambda_{1,0},\lambda_{0,1})$, it is clear  that $\widetilde \kappa^2\pG = 0$. Since moreover $\mu_{1,0}=\mu_{0,1}=0$, we obtain that $\mu_{2,0} =0$ and the result is proved for $\ell=0$. 

Given $\ell\in\mathbb N_0$, $\ell<q-1$,  assume that each $\mu_{i_x,i_y}$ for $1\leq i_x+i_y\leq \ell+2$ is a polynomial in $\mathbb C[\lambda_{1,0},\lambda_{0,1}]$ of degree at most equal to $i_x+i_y-2$.  Since $\mu_{0,\ell+3}$ and $\mu_{1,\ell+2}$ are both fixed to zero, they clearly are polynomials in $\mathbb C[\lambda_{1,0},\lambda_{0,1}]$ of degree at most equal to $\ell+1$. We then want to prove the result for all $\mu_{j_x+2,\ell+1-j_x}$ with $0\leq j_x\leq \ell+1$, and we will naturally proceed by induction on $j_x$.
 For $j_x=0$, since $\mu_{0,\ell+3}=0$, (\ref{inductionform}:) shows that $ \mu_{2,\ell+1}$ as a polynomial in $\mathbb C[\lambda_{1,0},\lambda_{0,1}]$ will simply be a multiple of $\mathsf{RHS}$, while thanks to the definition of $\widetilde \kappa^2$ with $\mu_{0,\ell+1}=\mu_{1,\ell+1}=\mu_{0,\ell+2} =0$ we have
 $$
\mathsf{RHS} =
 -\sum_{k_y=0}^{\ell+1} \frac1{(\ell+1-k_y)!} \dy^{\ell+1-k_y}  \kappa^2\pG \mu_{0,k_y}.
 $$
 So the induction hypothesis for $\ell$ together with $\mu_{0,0}=1$ show that the result holds for $ \mu_{2,\ell+1}$.
 Given  $j_x\in\mathbb N_0$, $j_x<\ell+1$, we now assume that each $\mu_{i_x,i_y}$ for $i_x+i_y=\ell+3$ as well as $i_x\leq j_x+2$
 is a polynomial in $\mathbb C[\lambda_{1,0},\lambda_{0,1}]$ of degree at most equal to $i_x+i_y-2$. The degree of $\mu_{j_x+3,\ell-j_x}$ at most equal to the maximum of:
 \begin{itemize}
 \item the degree of $\mu_{j_x+1,\ell-j_x+2}$, at most equal to $\ell+1$ from the induction hypothesis for $j_x$,
 \item the degree of $\lambda_{10}\mu_{j_x+2,\ell-j_x}$, at most equal to $\ell+1$ from the induction hypothesis for $\ell$,
 \item the degree of $\lambda_{01} \mu_{j_x+1,\ell-j_x+1}$, at most equal to $\ell+1$ from the induction hypothesis for $\ell$,
 \item the degree of $\left( \left( \lambda_{1,0} \right)^2 + \left( \lambda_{0,1} \right)^2\right)\mu_{j_x+1,\ell-j_x}$, at most equal to $\ell-1$ from the induction hypothesis for $\ell$,
 \item the degree of $\dx^{j_x+1}\dy^{\ell-j_x}\kappa^2 \pG \mu_{0,0}$, at most equal to $0$ since $\mu_{0,0}=1$,
 \item the degree of $\dx^{j_x+1-k_x}\dy^{\ell-j_x-k_y}\kappa^2 \pG \mu_{k_x,k_y}$ for $0\leq k_x\leq j_x+1$ and $0\leq k_y\leq \ell-k_x$, at most equal to $\ell-1$ from the induction hypothesis for $\ell$.
 \end{itemize}
So the degree of $\mu_{j_x+3,\ell-j_x}$ as a polynomial in $\mathbb C[\lambda_{1,0},\lambda_{0,1}]$ is as expected at most equal to $\ell+1$, which concludes the proof.
\end{proof}

\subsection*{ Step 2} 
This step focuses on identifying a reference set of functions, simpler to study than the GPWs.
From the choice of ansatz for amplitude-based GPWs, it is again clear that the reference case of classical PWs will play a fundamental role in the study of interpolation properties, similarly to the phase-based GPW case \cite{LMinterp}.

The normalization introduced before Definition \ref{def:GPWsp} is used to define the reference space of classical Plane Waves corresponding to the GPWs as follows.
\begin{dfn}
\label{def:cPWsp}
Assume that the point $\pG\in\mathbb R^2$ as well as the parameter $p$ are given.
For any set of $p$ angles $\mathrm A:=\{\theta_k \in [0,2\pi); 1\leq k \leq p\}$, 
for each angle we define the associated $\vec{d}_k:= \sqrt{-\kappa^2\pG} (\cos \theta_k,\sin\theta_k)$ and classical PW $H_k :=\exp \Big({\vec{d}_k} \cdot (\boldsymbol{\cdot}- \pG)\Big)$. We will denote the corresponding set of PW functions $\mathcal B^{ref}_{\mathrm A}:=\{H_k;\theta_k\in\mathrm A, 1\leq k\leq p\} $.
\end{dfn}
Even though as discussed earlier we will assume that $\kappa^2\pG\neq 0$, the following study will hold independently of the sign of $\kappa^2\pG$, however the $H_k$ functions are oscillating plane waves only if $\kappa^2\pG>0$, while they are exponentially decaying (or increasing) functions if  $\kappa^2\pG<0$.

\subsection*{ Step 3} 
Let's turn to the study of interpolation properties of this reference case. Here the reference case is the same for the amplitude-based GPWs as it was for the phase-based GPWs. The interpolation properties of the reference space $\mathcal B^{ref}_{\mathrm A}$ were proved in \cite{cd98}, even though they were not stated as a stand-alone result. The precise result that we will use in the following is only a part of that proof. It is simply presented here as a reminder.
%
%
\begin{lmm}
\label{lmm:rkMnC}
Assume that the point $\pG\in\mathbb R^2$ as well as the parameters $n\in\mathbb N$ are given.
For any set of $p=2n+1$ distinct angles $\mathrm A:=\{\theta_k \in [0,2\pi); 1\leq k \leq p\}$, 
we consider the matrix \eqref{eq:DefMat} for the reference set of PW $\mathcal B^{ref}_{\mathrm A}$,  denoted $\mathsf M^{C}_n$ as in \cite{LMinterp}.
Then as long as $\kappa^2\pG\neq 0$ the rank of this matrix is $rk\left(\mathsf M^{C}_n\right) = 2n+1$.
\end{lmm}
See Section 4.1 in \cite{IGS} for a comment on the need for $p$ to be at least equal to $2n+1$ to guarantee a rank equal to $2n+1$, in relation to properties of trigonometric functions.

\subsection*{ Step 4} 
In order to relate the new GPW case to the reference case, we introduce the matrix \eqref{eq:DefMat} for the new GPW set $\mathcal B_{\mathrm A,q}$,  denoted $\mathsf M^{G}_n$. To study its relation to the reference matrix $\mathsf M^{C}_n$, we will first express each of its entries, the new basis functions' derivatives evaluated at $\pG$ in terms of the reference basis functions' derivatives evaluated at $\pG$.

\begin{lmm}
\label{lmm:ders}
Assume that the point $\pG\in\mathbb R^2$ as well as the parameters $(p,q)\in\mathbb N^2$ are given and that $\kappa^2$ is a function of class $\mathcal C^{q-1}$ in a neighborhood of $\pG$.
For any set of $p$ distinct angles $\mathrm A:=\{\theta_k \in [0,2\pi); 1\leq k \leq p\}$, ranked in a given order, we consider the reference and GPW bases $\mathcal B^{ref}_{\mathrm A}$ and $\mathcal B_{\mathrm A,q}$, with their elements ranked in the same order.
For $(j_x,j_y)\in\mathbb N_0^2$ such that $j_x+j_y\leq q+1 $, there exists a polynomial $R_{(j_x,j_y)}\in\mathbb C[\lambda_{1,0},\lambda_{0,1}]$ with $\deg R_{(j_x,j_y)}<j_x+j_y$ such that for all $k\in\mathbb N$ with $k\leq p$ we have:
\begin{equation*}
\frac{1}{j_x!j_y!}\dx^{j_x} \dy^{j_y} G_k \pG
=
\frac{1}{j_x!j_y!}\dx^{j_x} \dy^{j_y}  H_k \pG
+
R_{(j_x,j_y)}(\lambda_{1,0},\lambda_{0,1}).
\end{equation*}
\end{lmm}

\begin{proof}
For the vector $\vec{d}_k:= \sqrt{-\kappa^2\pG} (\cos \theta_k,\sin\theta_k)$ associated to any angle $\theta_k\in\mathrm A$, the reference and new basis functions are respectively $H_k$ and $G_k$. We first notice that $G_k=Q\cdot H_k$ where $Q$ is the polynomial constructed via Algorithm \ref{Algo:AbGPW}. Then, since $\dx^{i_x} \dy^{i_y} Q\pG = i_x!i_y!\mu_{i_x,i_y}$  for any $(i_x,i_y)\in\mathbb N_0^2$ such that $i_x+i_y\leq q+1 $, the product rule shows that as long as $j_x+j_y\leq q+1 $ we have:
\begin{equation*}
\frac{1}{j_x!j_y!}\dx^{j_x} \dy^{j_y} G_k \pG
=
\sum_{i_x=0}^{j_x}
\sum_{i_y=0}^{j_y}
\frac{1}{(j_x-i_x)!(j_y-i_y)!}
\mu_{i_x,i_y}
\left( \lambda_{1,0} \right)^{j_x-i_x} \left( \lambda_{0,1} \right)^{j_y-i_y}.
\end{equation*}
In view of the normalization chosen in Algorithm \ref{Algo:AbGPW}, it gives for $(j_x,j_y)\in\mathbb N_0^2$ such that $j_x+j_y\leq q+1 $:
\begin{equation*}
\left\{
\begin{array}{l}
\displaystyle
\frac{1}{j_y!}\dy^{j_y} G_k \pG
=
\frac{1}{j_y!}
 \left( \lambda_{0,1} \right)^{j_y},\text{if }j_x= 0,
 \\
 \displaystyle
\frac{1}{j_y!}\dx \dy^{j_y} G_k \pG
=
\frac{1}{j_y!}
\left( \lambda_{1,0} \right)^1 \left( \lambda_{0,1} \right)^{j_y},\text{if }j_x= 1,
\\
\displaystyle
\frac{1}{j_x!j_y!}\dx^{j_x} \dy^{j_y} G_k \pG
=
\frac{1}{j_x!j_y!}
\left( \lambda_{1,0} \right)^{j_x} \left( \lambda_{0,1} \right)^{j_y}
+
\sum_{i_x=2}^{j_x}
\sum_{i_y=0}^{j_y}
\frac{ \left( \lambda_{0,1} \right)^{j_y-i_y}\left( \lambda_{1,0} \right)^{j_x-i_x}}{(j_x-i_x)!(j_y-i_y)!}
\mu_{i_x,i_y},
\text{if }j_x\geq 2.
\end{array}
\right.
\end{equation*}
On the other hand, $\dx^{j_x} \dy^{j_y} H_k \pG =\left( \lambda_{0,1} \right)^{j_y-i_y}\left( \lambda_{1,0} \right)^{j_x-i_x} $ for any $(j_x,j_y)\in\mathbb N_0^2$. So the result is proved for $j_x=0$ and $j_x=1$, while for $j_x\geq 2$ the result is a direct consequence of Lemma \ref{lmm:mus}.
\end{proof}

\begin{lmm}
\label{lmm:rkMn}
Assume that the point $\pG\in\mathbb R^2$ as well as the parameters $n\in\mathbb N$ are given and that $\kappa^2$ is a function of class $\mathcal C^{q-1}$ in a neighborhood of  $\pG$.
For any set of $p=2n+1$ distinct angles $\mathrm A:=\{\theta_k \in [0,2\pi); 1\leq k \leq p\}$, ranked in a given order, we consider the reference and GPW matrices, $\mathsf M^{C}_n$ and $\mathsf M^{G}_n$,  respectively defined for the reference and the GPW bases $\mathcal B^{ref}_{\mathrm A}$ and $\mathcal B_{\mathrm A,q}$ for any $q\geq\max(n-1,1)$, with their elements ranked in the same order.
There exists a non-singular matrix $\mathsf L_n\in\mathbb C^{(n+1)(n+2)/2\times (n+1)(n+2)/2}$ such that
\begin{equation*}
\mathsf M^{G}_n= \mathsf L _n\mathsf M^{C}_n.
\end{equation*}
As a consequence, as long as $\kappa^2\pG\neq 0$, $rk\left(\mathsf M^{G}_n\right) = 2n+1$.
\end{lmm}
\begin{proof}
The choice  $q\geq\max(n-1,1)$ guarantees that $n\leq q+1$, therefore the result of Lemma \ref{lmm:ders} holds in particular for all entries of the matrices, and it can be restated as: there  exits  a complex polynomial in two variables $R_{(j_x,j_y)}$ with $\deg R_{(j_x,j_y)}<j_x+j_y$ such that for all $k\in\mathbb N$ with $k\leq p$ we have
\begin{equation}
\label{eq:rowMn}
\left( \mathsf M^{G}_n \right)_{\frac{(j_x+j_y)(j_x+j_y+1)}{2}+j_y+1,k}
=
\left( \mathsf M^{C}_n \right)_{\frac{(j_x+j_y)(j_x+j_y+1)}{2}+j_y+1,k}
+
R_{(j_x,j_y)}\left(
\left( \mathsf M^{C}_n \right)_{2,k},
\left( \mathsf M^{C}_n \right)_{3,k}
\right).
\end{equation}
Moreover, from the definition of the reference basis functions $H_k$, we have 
$$\dx^{j_x} \dy^{j_y}  H_k \pG 
=\Big(\lambda_{1,0}\Big)^{j_y} \Big( \lambda_{0,1}\Big)^{j_x} 
=\ \dx H_k \pG^{j_x} \dy  H_k \pG^{j_y}$$ 
where $\left(\lambda_{1,0},\lambda_{0,1}\right)=\sqrt{-\kappa^2\pG} (\cos \theta_k,\sin\theta_k)$.
Hence, from the definition of the reference matrix $\mathsf M^{C}_n$, we can identify
\begin{equation*}
\left(\left( \mathsf M^{C}_n \right)_{2,k}\right)^{j_y}
\left(\left( \mathsf M^{C}_n \right)_{3,k}\right)^{j_x}
=
\Big(\mathsf M^C_n\Big)_{\frac{(j_x+j_y)(j_x+j_y+1)}{2}+j_y+1,k}.
\end{equation*}
  It is now clear that \eqref{eq:rowMn} is precisely stating that $ \mathsf M^{G}_n$'s row number $\frac{(j_x+j_y)(j_x+j_y+1)}{2}+j_y+1$ can be written as a linear combination of  $ \mathsf M^{C}_n$'s rows, more precisely it can be written as the sum of:
  \begin{itemize}
  \item once $ \mathsf M^{C}_n$'s row number $\frac{(j_x+j_y)(j_x+j_y+1)}{2}+j_y+1$,
  \item a linear combination of $ \mathsf M^{C}_n$'s row of index at most equal to $\frac{(j_x+j_y)(j_x+j_y+1)}{2}$.
  \end{itemize}
  Finally this is equivalent to the existence of a lower unitriangular matrix $\mathsf L_n$ such that $\mathsf M^{G}_n= \mathsf L _n\mathsf M^{C}_n$.
  
As a consequence, $rk\left(\mathsf M^{G}_n\right)= rk\left(\mathsf M^{C}_n\right)$. Hence, from Lemma \ref{lmm:rkMnC}, if $\kappa^2\pG\neq 0$ then $rk\left(\mathsf M^{G}_n\right)=2n+1$.
\end{proof}

\subsection*{Step 5} 
Finally we can pull the pieces together to study the interpolation properties of the space spanned by the new GPWs.
\begin{dfn}
Assume that the point $\pG\in\mathbb R^2$ as well as the parameters $(p,q)\in\mathbb N^2$ are given and that $\kappa^2$ is a function of class $\mathcal C^{q-1}$ in a neighborhood of $\pG$.
For any set of $p$ distinct angles $\mathrm A:=\{\theta_k \in [0,2\pi); 1\leq k \leq p\}$, ranked in a given order, 
%
we consider the GPW basis $\mathcal B_{\mathrm A,q}$ for any $q\geq\max(n-1,1)$. The complex vector spaced spanned by  $\mathcal B_{\mathrm A,q}$ will be denoted $\mathbb V_{\pG}^{\mathrm A,q}$.
\end{dfn}
In order to obtain the desired interpolation properties, it is therefore sufficient to pick the approximation parameter $q$ to be equal to $\max(n-1,1)$. Picking a higher value would guarantee the same properties, but would result in an unnecessary increase of the GPW construction's computational cost.
\begin{thm}
\label{thm:AbInterp}
Assume that the point $\pG\in\mathbb R^2$ as well as the parameters $n\in\mathbb N$ are given and that $\kappa^2$ is a function of class $\mathcal C^{n}$ in a neighborhood of $\pG$, with $\kappa^2\pG\neq 0$.
For any set of $p=2n+1$ distinct angles $\mathrm A:=\{\theta_k \in [0,2\pi); 1\leq k \leq p\}$, we consider the GPW space $\mathbb V_{\pG}^{\mathrm A,\max(n-1,1)}$.

For any solution $u$ of the PDE \eqref{eq:Helm} which is  of class $\mathcal C^{n}$ at $\pG$, there exists a GPW function $u_a\in\mathbb V_{\pG}^{\mathrm A,\max(n-1,1)}$ and a constant $C$ such that in a neighborhood of $\pG$:
\begin{equation}
\label{eq:AppProp}
\left\{
\begin{array}{l}
\displaystyle
\left| (u - u_a) (x,y) \right| \leq  C |(x,y)-\pG|^{n+1}\\
\displaystyle
\left| (\nabla u - \nabla u_a) (x,y) \right| \leq  C |(x,y)-\pG|^{n}
\end{array}
\right.
\end{equation}
\end{thm}
 Thanks to our preliminary Lemmas, the proof is identical to the one from \cite{LMinterp} for phase-based GPWs. We repeat the proof here for the sake of completeness. Even though the preliminary steps were studied under the assumption that $\kappa^2\pG\neq 0$, as it was mentioned in Section \ref{sec:constr}, amplitude-based GPWs with an appropriate normalization enjoy the same interpolation properties even if $\kappa^2\pG=0$.
\begin{proof}
It is sufficient for the Taylor expansion of a function $u_a\in\mathbb V_{\pG}^{\mathrm A,\max(n-1,1)}$ to match that of the solution $u$ to prove the theorem.

Any element of $\mathbb V_{\pG}^{\mathrm A,\max(n-1,1)}$ can be written as $\displaystyle \sum_{k=1}^{2n+1} X_{k}G_k$, and its Taylor expansion  matches that of the solution $u$ if and only if
\begin{equation}
\label{eq:MatVecSys}
\mathsf M^{G}_n \mathsf X = \mathsf U
\end{equation}
where the $k$th entry of $\mathsf X\in\mathbb C^{2n+1}$ is the coefficient $X_k$ while for all $(j_x,j_y)\in(\mathbb N_0)^2$ the $\frac{(j_x+j_y)(j_x+j_y+1)}{2}+j_y+1$th entry of $\mathsf U\in\mathbb C^{\frac{(n+1)(n+2)}{2}}$ is the Taylor expansion coefficient $\dx^{j_x} \dy^{j_y} u \pG/(j_x!j_y!)$. From Lemma \ref{lmm:rkMn} the matrix $\mathsf M^{G}_n\in\mathbb C^{\frac{(n+1)(n+2)}{2}\times (2n+1)}$ has maximal rank $2n+1$. Moreover, the range of $\mathsf M^{G}_n$ can be identified as
$$
\mathfrak K:= \left\{ (C_{j_x,j_y}) \in  \mathbb C^{\frac{(n+1)(n+2)}{2}}, \forall(j_x,j_y)\in\mathbb N^2, j_x+j_y\leq n-2, 
\phantom{\sum_{j=0}^{k_1} \frac{\dx^i \dy^j \beta\left(\overrightarrow g\right)}{i!j!} C_{k_1-j,k_2} (k_1+1)}
\right.
$$
$$
\left. (j_x+1)(j_x+2)C_{j_x+2,j_y} +(j_y+1)(j_y+2)C_{j_x,j_y+2}=-\sum_{i_x=0}^{j_x}\sum_{i_y=0}^{j_y} \frac{\dx^{i_x} \dy^{i_y} \kappa^2\pG}{i_x!i_y!} C_{j_x-i_x,j_y-i_y} \right\},
$$
and the right hand side $\mathsf U$ of \eqref{eq:MatVecSys} clearly belongs to $\mathfrak K$ as the function $u$ solves the PDE \eqref{eq:Helm}.

As a result there exists a solution $\mathsf X$ to the linear system  \eqref{eq:MatVecSys}, and the corresponding GPW function $u_a:=\displaystyle \sum_{k=1}^{2n+1} X_{k}G_k\in \mathbb V_{\pG}^{\mathrm A,\max(n-1,1)}$ is guaranteed to have the same Taylor expansion as $u$ up to order $n$. In other words, we have $(u-u_a)(x,y)=O\left(|(x,y)-\pG|^{n+1}\right)$, and this clearly implies the result \eqref{eq:AppProp}.
\end{proof}

\section{Extension beyond the Helmholtz equation}
\label{sec:Ext}
It is a natural question to consider how can this work, developed in the previous sections for the Helmholtz operator $-\Delta-\kappa^2(x,y)$, be extended to
other linear partial differential operators. This was the goal of \cite{IGS} for phase-based GPWs, where were considered operators of order $\mfn \geq 2$, in two dimensions, of the form
\begin{equation}
\label{eq:LinOp}
 \Lal := \sum_{\ell = 0}^{\mfn} \sum_{k=0}^\ell \alpha_{k,\ell-k}\left( x , y \right) \partial_x^k \partial_y^{\ell-k},
\end{equation}
with   $\alpha = \{\alpha_{k,\ell-k},(k,\ell)\in\mathbb N^2, 0\leq k\leq\ell\leq \mfn\}$ represents the set  of complex-valued coefficients. The situation is similar for amplitude-based GPWs.
On the one hand, the construction process of amplitude-based GPWs proposed earlier can be extended as-is to a large family of linear partial differential equations of the form \eqref{eq:LinOp}, under a simple assumption $\alpha_{M,0}\pG\neq 0$. This is simply because the key to the construction algorithm does not lie in any of the normalization choices, but rather in Formula \eqref{inductionform}.
On the other hand, the interpolation properties strongly rely on the normalization, as we have seen that   the fact for  $ \left( \lambda_{1,0} \right)^2 + \left( \lambda_{0,1} \right)^2$ to be constant was fundamental as early as in {\bf Step 1}.

In this section, instead of considering general operators \eqref{eq:LinOp} as in \cite{IGS}, we will focus on operators of second order allowing for anisotropy in the first and second order terms, written under the form
\begin{equation}
\label{eq:GenOp}
\mathcal L:= \nabla \cdot A(x,y) \nabla + V(x,y) \cdot \nabla +  s(x,y),
\end{equation} 
with $A$ a matrix-valued function, $V$ a vector-valued function, and $s$ a scalar-valued function, under the assumption that the matrix $A\pG$ is real symmetric with non-zero eigenvalues.
We will point out how the construction process and the proof of interpolation properties can both be adapted to these operators, see respectively Subsections \ref{ssec:ExtConstr} and \ref{ssec:ExtInterp}.
However, similarly to the work presented in \cite{IGS}, this work extends to certain operators of higher order like \eqref{eq:LinOp} -- under the same hypothesis.

\subsection{The construction process}
\label{ssec:ExtConstr}
In the construction process, the crucial point lies in the identification of linear subsystems thanks to the layer structure. 
Defining 
$$
\widetilde V  := 
\begin{pmatrix}
\dx A_{11} +2\lambda_{1,0}A_{11}+A_{12}\lambda_{0,1} + \dy A_{21} + A_{21} \lambda_{0,1}+V_1
\\
\dy A_{22} +2\lambda_{0,1}A_{22}+A_{21}\lambda_{1,0} + \dx A_{12} + A_{12} \lambda_{1,0}+V_2
\end{pmatrix},
$$
$$
\begin{array}{rl}
\widetilde s :=  &
s+\lambda_{1,0} + \lambda_{0,1} 
+(\dx A_{11} + \dx A_{21})\lambda_{1,0} + (\dx A_{12} +\dy A_{22})\lambda_{0,1} \\
&+ A_{11}\lambda_{1,0}^2 +(A_{12}+A_{21}) \lambda_{1,0}\lambda_{0,1} + A_{22} \lambda_{0,1}^2,
\end{array}
$$
one can easily verify that the equivalent to Problem \eqref{pb:AbGPW} reads
\begin{equation}
\label{eq:toTE}
\left\{
\begin{array}{l}
\text{Find } \left(Q, (\lambda_{10} , \lambda_{01} )\right)\in\mathbb C[X,Y]\times\mathbb C^2 \text{ such that }\\
A_{11}\dx^2 Q (x,y) +(A_{12}+A_{21})\dx\dy Q (x,y) +A_{22}\dy^2 Q (x,y) 
+\widetilde V(x,y)\cdot \nabla Q(x,y)
+\widetilde s(x,y)Q(x,y) \\
\qquad\qquad\qquad\qquad\qquad\qquad\qquad\qquad\qquad\qquad\qquad\qquad\qquad\qquad\qquad= O(|(x,y)-\pG|^q)\\
G(x,y) := Q(x,y)\exp \Big( \lambda_{10} (x-\xG) + \lambda_{01} (y-\yG) \Big)
\end{array}
\right.
\end{equation}
where it is important to keep in mind that the second order term coefficients are not constant here.  For a given $ (\lambda_{10} , \lambda_{01} )\in\mathbb C^2$, as in Subsection \ref{ssec:Lsys}, the linear system's unknowns are still the polynomial coefficients of $Q$ while the  linear system's equations are still the Taylor expansion coefficients of \eqref{eq:toTE}. The degree of $Q$, for the same reason as earlier, is set to $\deg Q = q+1$.
For a more compact notation we will write $A^c:=A\pG$.
Instead of the Laplacian defined on spaces of homogeneous polynomials, it is the operator $\nabla \cdot A^c\nabla $ defined on spaces of homogeneous polynomials which is here key to the layer structure. As a consequence, the well-posedness relies on the size of the kernel of this operator, which is again equal to $2$ according to the identity
\begin{equation*}
\begin{array}{l}
 \displaystyle
\nabla \cdot A^c\nabla \left[\sum_{i=0}^{\ell+2} p_i X^iY^{\ell+2-i}\right]
\\ \displaystyle
=
\sum_{i=0}^{\ell} \Big(
    (i+2)(i+1)           A_{11}^c p_{i+2} 
+ (i+1)(\ell+1-i)      (A_{12}^c+A_{21}^c) p_{i+1}
+ (\ell+2-i)(\ell+1-i)A_{22}^cp_i
\Big) X^iY^{\ell-i}.
\end{array}
\end{equation*}

Therefore the linear system is also has a solution, and an algorithm following the same stages as  \ref{Algo:AbGPW}, with updated formulas for $\mathsf{RHS}$ and $ \mu_{j_x+2,\ell-j_x} $, constructs an amplitude-based GPW for any $\vec{d}=(\lambda_{1,0},\lambda_{0,1})\in\mathbb C^2$ which satisfies the local approximation property in the vicinity of $\pG$
\begin{equation*}
(\nabla \cdot A(x,y) \nabla + V(x,y) \cdot \nabla +  s(x,y))G_{\vec d}(x,y) = O(|(x,y)-\pG|^q),
\end{equation*}
indepedently of the normalization.

\subsection{Interpolation properties}
\label{ssec:ExtInterp}
The impact of the normalization on the interpolation properties appears in Section \ref{sec:int} when we express the coefficients $\{ \mu_{i_x,i_y},0\leq i_x+i_y\leq q+1 \}$ of $Q$ in terms of the phase coefficients $(\lambda_{1,0},\lambda_{0,1})$. In Lemma \ref{lmm:mus}'s proof, the assumption that $(\lambda_{1,0})^2+(\lambda_{0,1})^2$ is fixed is crucial to show that $\mu_{2,0}$ can be expressed as a constant polynomial in $\mathbb C[\lambda_{1,0},\lambda_{0,1}]$. In the case of the generalized operator \eqref{eq:GenOp}, the expression for the unknown $\mu_{20}$ corresponding to \eqref{eq:mu20}  reads
$$
\mu_{2,0} 
= 
-\frac1{2A_{11}\pG}\left(
\widetilde V_1\pG\mu_{1,0} +\widetilde V_2\pG\mu_{0,1} + \widetilde s\pG \mu_{0,0}
\right)
$$
Since $\mu_{1,0}=\mu_{0,1}=0$ while $\mu_{0,0}=1$, we can focus our attention on the last term. From the definition of $\widetilde s$ we can consider $\widetilde s \pG$ as an element $\mathbb C[\lambda_{1,0},\lambda_{0,1}]$:
$$
\begin{array}{rl}
\widetilde s \pG = &
s\pG+\lambda_{1,0} + \lambda_{0,1} 
+(\dx A_{11} + \dx A_{21})\pG\lambda_{1,0} + (\dx A_{12} +\dy A_{22})\pG\lambda_{0,1} \\&
+ A_{11}\pG\lambda_{1,0}^2 +(A_{12}+A_{21})\pG \lambda_{1,0}\lambda_{0,1} + A_{22} \pG\lambda_{0,1}^2.
\end{array}
$$
The choice of normalization for $(\lambda_{1,0},\lambda_{0,1})$ then determines how, in turn, $\mu_{2,0} $ can be expressed as an element of  $\mathbb C[\lambda_{1,0},\lambda_{0,1}]$.
For the Helmholtz case, since $A_{12} = A_{21}=0$ and $A_{11}=A_{22}=1$, the second degree terms were reduced to $(\lambda_{1,0})^2+(\lambda_{0,1})^2$ and there were no first degree terms. The normalization assumption that $(\lambda_{1,0})^2+(\lambda_{0,1})^2$ was fixed therefore resulted in the expression of $\mu_{2,0}$ as a constant polynomial in $\mathbb C[\lambda_{1,0},\lambda_{0,1}]$.  But here the first degree terms are non zero unless the matrix $A\pG$ is constant, and there are three second degree terms. However, since the matrix $A\pG$ is assumed to be real symmetric with non-zero eigenvalues, denoted here $\gamma_1,\gamma_2$, it can be diagonalized by an orthogonal matrix: there exists an orthogonal matrix $P$ such that $\begin{pmatrix}\gamma_1&0\\0&\gamma_2\end{pmatrix}=P A\pG P^T$, and the second degree terms can then be rewritten as
$$
\begin{array}{l}
A_{11}\pG\lambda_{1,0}^2 +(A_{12}+A_{21})\pG \lambda_{1,0}\lambda_{0,1} + A_{22} \pG\lambda_{0,1}^2 \\
=\gamma_1 \left( P_{11}\lambda_{1,0}+P_{12}\lambda_{0,1} \right)^2
+\gamma_2 \left( P_{21}\lambda_{1,0}+P_{22}\lambda_{0,1}  \right)^2.
\end{array}
$$
The normalization assumption that 
\begin{equation}
\label{eq:GenNorm}
\begin{pmatrix} \lambda_{1,0}\\\lambda_{0,1}\end{pmatrix}
\propto 
P^T \begin{pmatrix}1/\sqrt{\gamma_1}&0\\0&1/\sqrt{\gamma_2}\end{pmatrix}
\begin{pmatrix} \cos\theta\\\sin\theta\end{pmatrix}
\end{equation}
then guarantees that $\mu_{2,0}$ can be written as a polynomial of degree at most equal to $1$ in  $\mathbb C[\lambda_{1,0},\lambda_{0,1}]$. So the normalization is different from that of the Helmholtz case, although if $A$ is the identity matrix and $V$ is the zero vector then \eqref{eq:GenNorm} reduces to the Helmholts PW normalization.

Starting from this point and following a similar reasoning as that of Lemma \ref{lmm:mus}'s proof, we can prove that each $\mu_{i_x,i_y}$ for $i_x+i_y\geq 1$ can be expressed as a polynomial in $\mathbb C[\lambda_{1,0},\lambda_{0,1}]$ of degree at most equal to $i_x+i_y-1$. 

The roadmap's second and third steps are independent of the GPW normalization choice as the reference case is still the classical PW case, with propagation direction 
$( \cos\theta,\sin\theta)$  
but without any restriction on the wave number, that is functions  $H_k :=\exp \Big({\vec{d}_k} \cdot (\boldsymbol{\cdot}- \pG)\Big)$ with some $\vec{d}_k\propto (\cos \theta_k,\sin\theta_k)$ .

There is an important consequence of the normalization to comment concerning the roadmap's fourth step. 
The functions  $F_k :=\exp \Big({\vec{e}_k} \cdot (\boldsymbol{\cdot}- \pG)\Big)$ with $\vec{e}_k:= (\lambda_{1,0},\lambda_{0,1})$  with the normalization \eqref{eq:GenNorm} for $\theta=\theta_k$ are a natural intermediate between the GPWs and the classical PWs, useful to write the GPWs as $G_k=Q\cdot F_k$. While in the Helmholtz case it is clear that $H_k=F_k$ for the appropriate choice of wave number, this is not the case as soon as the matrix $A$ is distinct from the identity. The pending result to Lemma \ref{lmm:ders} would then relate the derivatives of $G_k$s to derivatives of $F_k$: any derivative of order $j_x+j_y$ of the difference $G_k-F_k$ could be expressed as a polynomial of degree smaller than $j_x+j_y$ in $\mathbb C[\lambda_{1,0},\lambda_{0,1}]$. The only impact of the normalization on the proof is through Lemma \ref{lmm:mus}, and $\mu_{i_x,i_y}$ being expressed as a polynomial of degree at most equal to $i_x+i_y-1$ is sufficient to conclude.
The fact that $F_k\neq H_k$ then requires an additional step to relate the matrices $\mathsf M^{G}_n$ and $\mathsf M^{C}_n$, thanks to the introduction of the matrix $\mathsf M^{F}_n$ corresponding to the intermediate set of functions $\{F_k, k\in\mathbb N, k\leq p\}$. Relating $\mathsf M^{G}_n$ to $\mathsf M^{F}_n$ follows from the previous comment, just as in Lemma \ref{lmm:rkMn}. Relating $\mathsf M^{F}_n$ to  $\mathsf M^{C}_n$  is straightforward as can be seen from Lemma 7 in \cite{IGS}. This provides again all the pieces to prove that the rank of $\mathsf M^{G}_n$ is $2n+1$ as expected.

Finally, one more time, the fifth step is a direct consequence of the preliminary work. So the proof of Theorem \ref{thm:AbInterp} does not need to be adapted, and the same interpolation properties hold for the operator \eqref{eq:GenOp} as for the Helmholtz operator, under the assumptions that
\begin{itemize}
\item  the matrix $A\pG$ is real symmetric with non-zero eigenvalues,
\item the normalization of $(\lambda_{1,0},\lambda_{0,1})$ is chosen according to \eqref{eq:GenNorm}.
\end{itemize}

\section{Numerical results}
\label{sec:NR}
This section is dedicated to illustrating the interpolation properties of amplitude-based GPWs, for the Helmholtz equation and beyond. The most natural aspect to discuss is the expected high order convergence, but we will also emphasize aspects of conditioning, 
  as well as a comparison with phase-based GPWs in terms of pre-asymptotic behavior. 

The testing procedure follows the structure of the theoretical part of this article. The parameters are set according to the theorem: $p=2n+1$, $q=\max(n-1,1)$, for increasing values of $n$. 
A set of GPWs with appropriate normalization is constructed following Algorithm \ref{Algo:AbGPW}, and the GPW approximation $u_a$'s coefficients in the GPW basis are computed following the proof of Theorem \ref{thm:AbInterp}. However, rather than considering a single point $\pG$, we will consider 50 random points distributed in a given domain $\Omega$, as we keep in mind that the GPWs are meant to be local basis for a problem set on the full domain.

In order to illustrate precisely the first estimate from Theorem \ref{thm:AbInterp}, 
$$
\left| (u - u_a) (x,y) \right| \leq  C |(x,y)-\pG|^{n+1},
$$
we investigate the $h$ convergence, i.e in the regime $h:= |(x,y)-\pG|$ approaching $0$, of the difference between the exact solution $u$ to be approximated and the constructed GPW approximation $u_a$. The approximate function $u_a$ itself is independent of $h$ but depends on $n$.
The error reported here is an estimate of the $L^\infty$ norm of the difference $u-u_a$ on the circle centered at $\pG$ of radius $h$. This is different from the norm reported in previous work, where we reported the  $L^\infty$ norm of the difference $u-u_a$ on the disk centered at $\pG$ of radius $h$.
We then report the largest error among errors obtained at each of the 50 random points $\pG$.
 According to the theorem, for each value of $n$ we expect to observe convergence of order $n+1$. The constant $C$ depends both on $n$ and on $\pG$,

All the numerical experiments presented in this article were computed with the same set of normalization angles, $\mathrm A:=\left\{\theta_k =\frac{2(k-1)\pi}{p}+\frac{\pi}{6}; 1\leq k \leq p\right\}$, in order to avoid any particular alignment of the basis functions with the coordinate axes.

\subsection{List of test cases} Each test case consists of a variable-coefficient PDE \eqref{eq:Helm} or \eqref{eq:GenOp}, a given domain, and an exact solution $u$ to this PDE. 
It is straightforward to verify that these operators satisfy the hypothesis described in section \ref{sec:Ext}.

\begin{center}
\begin{tabular}{|c|c|c|c|}
\hline
ref. &Operator                    & Domain & Exact solution\\\hline\hline
Ae&
 $\Delta -(x-1)$&
 $[-2,2]\times[-2,2]$&
$ u(x,y) = Ai(x) \exp i y$
 \\\hline
 Ac&
 $\Delta -(x-1)$&
 $[-2,2]\times[-2,2]$&
$ u(x,y) = Ai(x) \cos  y$
 \\\hline
 A+&
 $\Delta -2(x+y)$&
 $[-2,2]\times[-2,2]$&
$ u(x,y) = Ai(x+ y)$
 \\\hline
 cs&
 $\dx^2+.2\cos x\sin y\dx\dy-2\dy^2 +( .2\sin x \cos y -1  )$&
 $[-1,1]\times[-1,1]$&
 $u(x,y) = \cos x \sin y$
 \\\hline
 ey&
 $\Delta+1$ & 
 $[-1,1]\times[0,2\pi]$ &
 $u(x,y) = \exp( i y)$ 
 \\\hline
 Jc&
 $x^2\Delta -x\dx+\cos y\dy-(1-2x^2-\sin y)$ &
 $[1,5]\times[0, 2\pi]$ &
 $u(x,y) =J_1(x)\cos y$
 \\\hline
 JJ&
 $x^2\dx^2+y^2\dy^2+x\dx+y\dy+(x^2+y^2-1)$&
 $[1,3]\times[0,3]$&
 $u(x,y)=J_0(x)J_1(y)$
 \\\hline
\end{tabular}
\end{center}

\subsection{Higher order convergence}
Figure \ref{fig:ey} displays the results obtained for the ey test case.  In this case, the operator is the Helmholtz operator with a constant wave number. The GPWs constructed by Algorithm \ref{Algo:AbGPW} are then exactly the classical PWs, as it is the case with phase-based GPWs. For each value of $n$, from $1$ to $20$, the order of convergence observed is the order obtained in the theorem. The error decreases until it reaches machine precision, and as $h$ keeps decreasing beyond this point the error remains of the same magnitude.
\begin{figure}
\begin{center}
\begin{tikzpicture}
\begin{loglogaxis}[height=10cm,xlabel=$h$, ylabel=max error on disks of radius h,
legend pos=outer north east,
xmin=10^(-7),xmax=20,ymin=10^(-16),ymax=10^(0),
xtick={1,0.01,.0001,.000001,.00000001},
very thick,cycle list name=nequalstwenty,grid=major]
\addplot table [x=h, y=errAn1]{./GPWBAE2D2OCaseexpikapy50pts.dat};
\addlegendentry{$n=1$}
\addplot table [x=h, y=errAn2]{./GPWBAE2D2OCaseexpikapy50pts.dat};
\addlegendentry{$n=2$}
\addplot table [x=h, y=errAn3]{./GPWBAE2D2OCaseexpikapy50pts.dat};
\addlegendentry{$n=3$}
\addplot table [x=h, y=errAn4]{./GPWBAE2D2OCaseexpikapy50pts.dat};
\addlegendentry{$n=4$}
\addplot table [x=h, y=errAn5]{./GPWBAE2D2OCaseexpikapy50pts.dat};
\addlegendentry{$n=5$}
\addplot table [x=h, y=errAn6]{./GPWBAE2D2OCaseexpikapy50pts.dat};
\addlegendentry{$n=6$}
\addplot table [x=h, y=errAn7]{./GPWBAE2D2OCaseexpikapy50pts.dat};
\addlegendentry{$n=7$}
\addplot table [x=h, y=errAn8]{./GPWBAE2D2OCaseexpikapy50pts.dat};
\addlegendentry{$n=8$}
\addplot table [x=h, y=errAn9]{./GPWBAE2D2OCaseexpikapy50pts.dat};
\addlegendentry{$n=9$}
\addplot table [x=h, y=errAn10]{./GPWBAE2D2OCaseexpikapy50pts.dat};
\addlegendentry{$n=10$}
\addplot table [x=h, y=errAn11]{./GPWBAE2D2OCaseexpikapy50pts.dat};
\addlegendentry{$n=11$}
\addplot table [x=h, y=errAn12]{./GPWBAE2D2OCaseexpikapy50pts.dat};
\addlegendentry{$n=12$}
\addplot table [x=h, y=errAn13]{./GPWBAE2D2OCaseexpikapy50pts.dat};
\addlegendentry{$n=13$}
\addplot table [x=h, y=errAn14]{./GPWBAE2D2OCaseexpikapy50pts.dat};
\addlegendentry{$n=14$}
\addplot table [x=h, y=errAn15]{./GPWBAE2D2OCaseexpikapy50pts.dat};
\addlegendentry{$n=15$}
\addplot table [x=h, y=errAn16]{./GPWBAE2D2OCaseexpikapy50pts.dat};
\addlegendentry{$n=16$}
\addplot table [x=h, y=errAn17]{./GPWBAE2D2OCaseexpikapy50pts.dat};
\addlegendentry{$n=17$}
\addplot table [x=h, y=errAn18]{./GPWBAE2D2OCaseexpikapy50pts.dat};
\addlegendentry{$n=18$}
\addplot table [x=h, y=errAn19]{./GPWBAE2D2OCaseexpikapy50pts.dat};
\addlegendentry{$n=19$}
\addplot table [x=h, y=errAn20]{./GPWBAE2D2OCaseexpikapy50pts.dat};
\addlegendentry{$n=20$}
\addplot[dotted] coordinates {(2*10^-6, 10^-10) (2*10^-3, 10^-4)};
\addlegendentry{order 2}
\end{loglogaxis}
\end{tikzpicture}
\end{center}
\caption{Convergence results for the ey test case, for $n$ from $1$ to $20$. For each value of $n$, the expected order of convergence, namely $n+1$, is observed and the error decreases until it reaches machine precision.}
\label{fig:ey}
\end{figure}
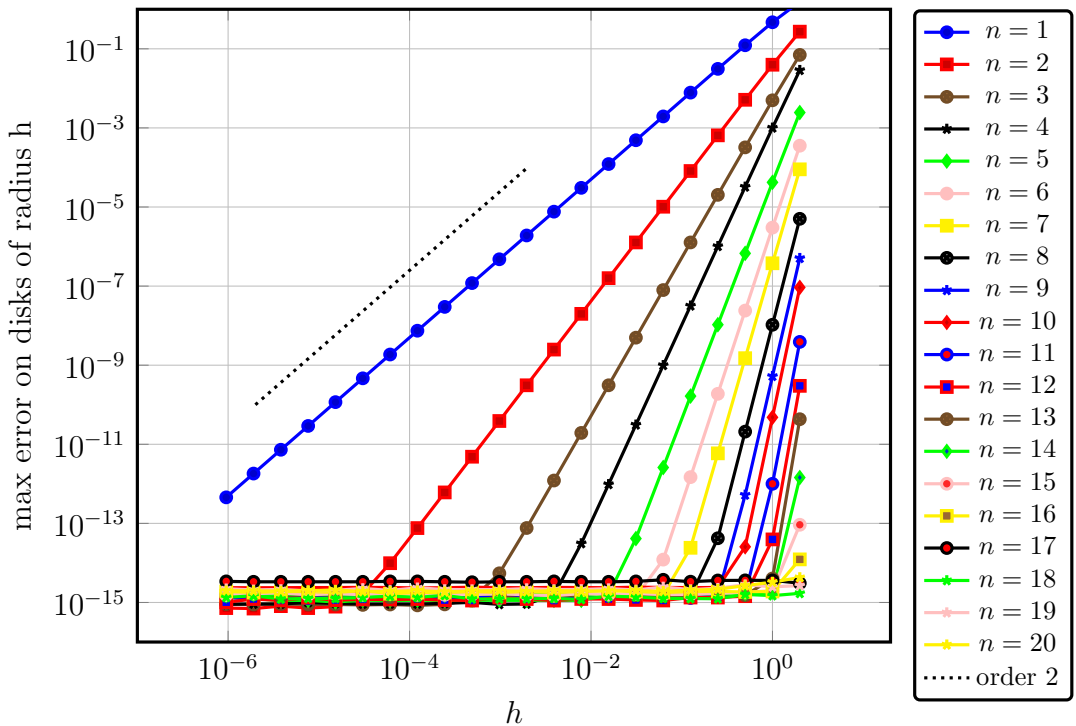

Figure \ref{fig:cond} displays the results obtained for four different test cases: Ae, Ac, A+ and cs. The first three correspond to the Helmholtz equation with a variable wave number, the variable wave number changes sign within of each of the associated domain, and the exact solutions oscillate for $\kappa^2>0$ and decay exponentially for $\kappa^2<0$. The last case corresponds to a more general operator \eqref{eq:GenOp}, with anisotropy in both the second and first order terms.  For each value of $n$, from $1$ to $20$, the order of convergence observed is the order obtained in the theorem. However, the matrix $\mathsf M_n^G$'s condition number increases with the value of $n$, so the coefficients $X_k$ of the linear combination $\displaystyle u_a\sum_{k=1}^{2n+1} X_{k}G_k$ are computed with decreasing accuracy as $n$ increases. This is illustrated by the behavior of the errors in Figure \ref{fig:cond}: the error's decrease is limited by the accuracy of the coefficient $X_k$, and as $h$ keeps decreasing beyond this point the error remains of the same magnitude. 
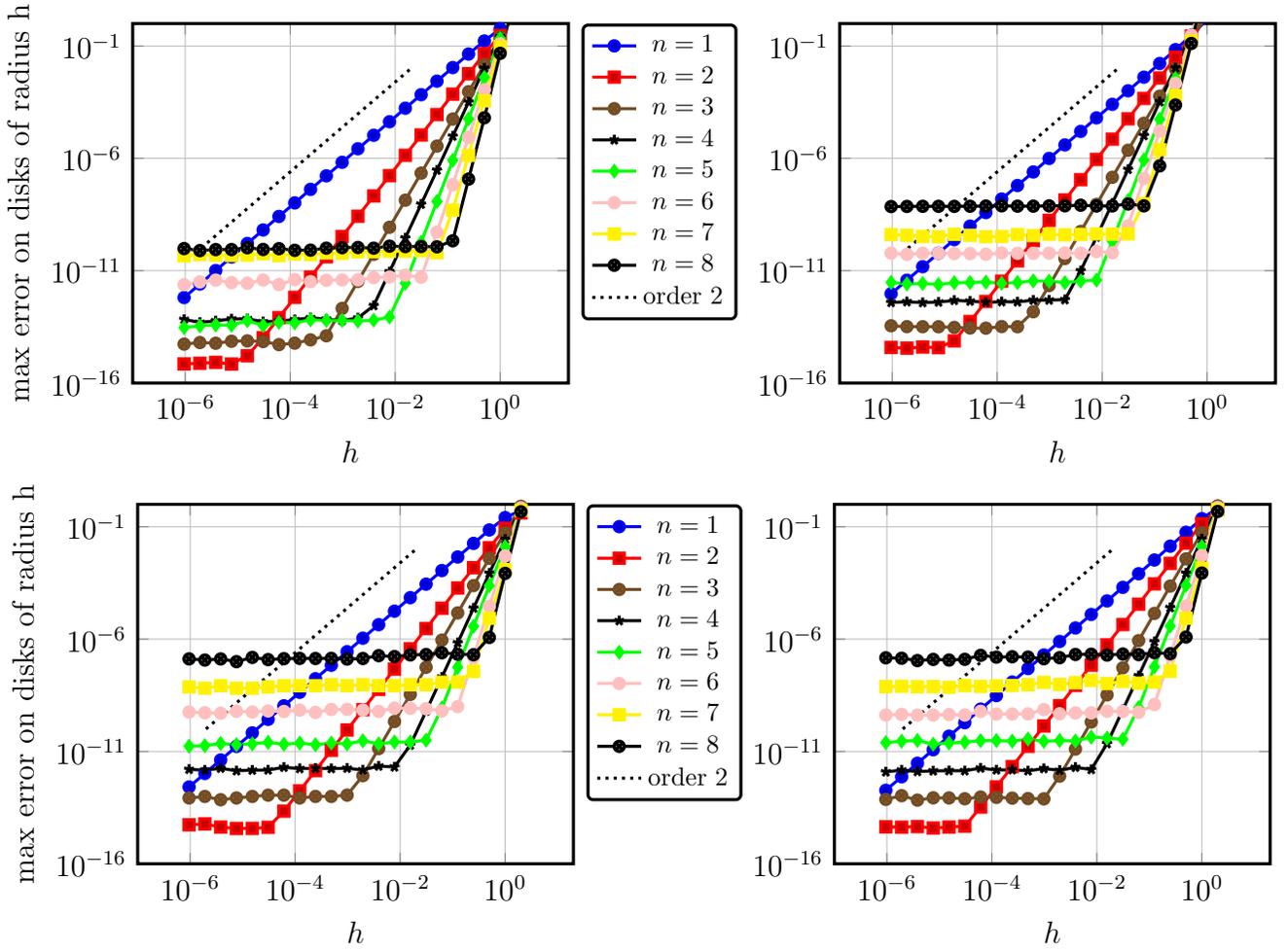
\begin{figure}
\begin{center}
\begin{tikzpicture}
\begin{loglogaxis}[height=6.5cm,xlabel=$h$, ylabel=max error on disks of radius h,
legend pos=outer north east,
xmin=10^(-7),xmax=20,ymin=10^(-16),ymax=10^(0),
xtick={1,0.01,.0001,.000001,.00000001},
very thick,cycle list name=nequalstwenty,grid=major]
\addplot table [x=h, y=errAn1]{./GPWBAE2D2OCaseAixpy50pts.dat};
\addlegendentry{$n=1$}
\addplot table [x=h, y=errAn2]{./GPWBAE2D2OCaseAixpy50pts.dat};
\addlegendentry{$n=2$}
\addplot table [x=h, y=errAn3]{./GPWBAE2D2OCaseAixpy50pts.dat};
\addlegendentry{$n=3$}
\addplot table [x=h, y=errAn4]{./GPWBAE2D2OCaseAixpy50pts.dat};
\addlegendentry{$n=4$}
\addplot table [x=h, y=errAn5]{./GPWBAE2D2OCaseAixpy50pts.dat};
\addlegendentry{$n=5$}
\addplot table [x=h, y=errAn6]{./GPWBAE2D2OCaseAixpy50pts.dat};
\addlegendentry{$n=6$}
\addplot table [x=h, y=errAn7]{./GPWBAE2D2OCaseAixpy50pts.dat};
\addlegendentry{$n=7$}
\addplot table [x=h, y=errAn8]{./GPWBAE2D2OCaseAixpy50pts.dat};
\addlegendentry{$n=8$}
\addplot[dotted] coordinates {(2*10^-6, 10^-10) (2*10^-2, 10^-2)};
\addlegendentry{order 2}
\end{loglogaxis}
\end{tikzpicture}
\begin{tikzpicture}
\begin{loglogaxis}[height=6.5cm,xlabel=$h$, 
xmin=10^(-7),xmax=20,ymin=10^(-16),ymax=10^(0),
xtick={1,0.01,.0001,.000001,.00000001},
very thick,cycle list name=nequalstwenty,grid=major]
\addplot table [x=h, y=errAn1]{./GPWBAE2D2OCasecosxsiny50pts.dat};
\addplot table [x=h, y=errAn2]{./GPWBAE2D2OCasecosxsiny50pts.dat};
\addplot table [x=h, y=errAn3]{./GPWBAE2D2OCasecosxsiny50pts.dat};
\addplot table [x=h, y=errAn4]{./GPWBAE2D2OCasecosxsiny50pts.dat};
\addplot table [x=h, y=errAn5]{./GPWBAE2D2OCasecosxsiny50pts.dat};
\addplot table [x=h, y=errAn6]{./GPWBAE2D2OCasecosxsiny50pts.dat};
\addplot table [x=h, y=errAn7]{./GPWBAE2D2OCasecosxsiny50pts.dat};
\addplot table [x=h, y=errAn8]{./GPWBAE2D2OCasecosxsiny50pts.dat};
\addplot[dotted] coordinates {(2*10^-6, 10^-10) (2*10^-2, 10^-2)};
\end{loglogaxis}
\end{tikzpicture}
\begin{tikzpicture}
\begin{loglogaxis}[height=6.5cm,xlabel=$h$, ylabel=max error on disks of radius h,
legend pos=outer north east,
xmin=10^(-7),xmax=20,ymin=10^(-16),ymax=10^(0),
xtick={1,0.01,.0001,.000001,.00000001},
very thick,cycle list name=nequalstwenty,grid=major]
\addplot table [x=h, y=errAn1]{./GPWBAE2D2OCaseAixcosy50pts.dat};
\addlegendentry{$n=1$}
\addplot table [x=h, y=errAn2]{./GPWBAE2D2OCaseAixcosy50pts.dat};
\addlegendentry{$n=2$}
\addplot table [x=h, y=errAn3]{./GPWBAE2D2OCaseAixcosy50pts.dat};
\addlegendentry{$n=3$}
\addplot table [x=h, y=errAn4]{./GPWBAE2D2OCaseAixcosy50pts.dat};
\addlegendentry{$n=4$}
\addplot table [x=h, y=errAn5]{./GPWBAE2D2OCaseAixcosy50pts.dat};
\addlegendentry{$n=5$}
\addplot table [x=h, y=errAn6]{./GPWBAE2D2OCaseAixcosy50pts.dat};
\addlegendentry{$n=6$}
\addplot table [x=h, y=errAn7]{./GPWBAE2D2OCaseAixcosy50pts.dat};
\addlegendentry{$n=7$}
\addplot table [x=h, y=errAn8]{./GPWBAE2D2OCaseAixcosy50pts.dat};
\addlegendentry{$n=8$}
\addplot[dotted] coordinates {(2*10^-6, 10^-10) (2*10^-2, 10^-2)};
\addlegendentry{order 2}
\end{loglogaxis}
\end{tikzpicture}\begin{tikzpicture}
\begin{loglogaxis}[height=6.5cm,xlabel=$h$, 
xmin=10^(-7),xmax=20,ymin=10^(-16),ymax=10^(0),
xtick={1,0.01,.0001,.000001,.00000001},
very thick,cycle list name=nequalstwenty,grid=major]
\addplot table [x=h, y=errAn1]{./GPWBAE2D2OCaseAixeiy50pts.dat};
\addplot table [x=h, y=errAn2]{./GPWBAE2D2OCaseAixeiy50pts.dat};
\addplot table [x=h, y=errAn3]{./GPWBAE2D2OCaseAixeiy50pts.dat};
\addplot table [x=h, y=errAn4]{./GPWBAE2D2OCaseAixeiy50pts.dat};
\addplot table [x=h, y=errAn5]{./GPWBAE2D2OCaseAixeiy50pts.dat};
\addplot table [x=h, y=errAn6]{./GPWBAE2D2OCaseAixeiy50pts.dat};
\addplot table [x=h, y=errAn7]{./GPWBAE2D2OCaseAixeiy50pts.dat};
\addplot table [x=h, y=errAn8]{./GPWBAE2D2OCaseAixeiy50pts.dat};
\addplot[dotted] coordinates {(2*10^-6, 10^-10) (2*10^-2, 10^-2)};
\end{loglogaxis}
\end{tikzpicture}
\end{center}
\caption{Convergence results for the  A+ (top left), cs (top right), Ac (bottom left) and Ae (bottom right) test cases, for $n$ from $1$ to $10$. While the expected order of convergence is observed in each case, the matrix $\mathsf M_n^G$'s ill-conditioning increases together with the value of $n$ which limits the accuracy of the corresponding GPW approximation.}
\label{fig:cond}
\end{figure}

Preliminary results suggest that the normalization chosen to compute the GPWs has a strong impact on $\mathsf M_n^G$'s condition number, as illustrated in Figure \ref{fig:csPWnorm}. These results were produced with a different normalization of the GPWs, namely $\vec{d}\propto (\cos\theta,\sin\theta)$ instead of the normalization \eqref{eq:GenNorm}. Note for instance that for $n=8$, the threshold shifts from approximately $10^{-8}$ in Figure \ref{fig:cond} (top right) to less than $10^{-12}$ in Figure  \ref{fig:csPWnorm}. However, we would like to emphasize that this simpler normalization provides better results for the cs test case, while it does not impact significantly the other three test cases presented in Figure \ref{fig:cond}. Further investigation in this direction is the topic of ongoing research.
\begin{figure}
\begin{center}
\begin{tikzpicture}
\begin{loglogaxis}[height=10cm,xlabel=$h$, ylabel=max error on disks of radius h,
legend pos=outer north east,
xmin=10^(-7),xmax=20,ymin=10^(-16),ymax=10^(0),
xtick={1,0.01,.0001,.000001,.00000001},
very thick,cycle list name=nequalstwenty,grid=major]
\addplot table [x=h, y=errAn1]{./GPWBAE2D2OCasecosxsiny50ptsNormAsPW.dat};
\addlegendentry{$n=1$}
\addplot table [x=h, y=errAn2]{./GPWBAE2D2OCasecosxsiny50ptsNormAsPW.dat};
\addlegendentry{$n=2$}
\addplot table [x=h, y=errAn3]{./GPWBAE2D2OCasecosxsiny50ptsNormAsPW.dat};
\addlegendentry{$n=3$}
\addplot table [x=h, y=errAn4]{./GPWBAE2D2OCasecosxsiny50ptsNormAsPW.dat};
\addlegendentry{$n=4$}
\addplot table [x=h, y=errAn5]{./GPWBAE2D2OCasecosxsiny50ptsNormAsPW.dat};
\addlegendentry{$n=5$}
\addplot table [x=h, y=errAn6]{./GPWBAE2D2OCasecosxsiny50ptsNormAsPW.dat};
\addlegendentry{$n=6$}
\addplot table [x=h, y=errAn7]{./GPWBAE2D2OCasecosxsiny50ptsNormAsPW.dat};
\addlegendentry{$n=7$}
\addplot table [x=h, y=errAn8]{./GPWBAE2D2OCasecosxsiny50ptsNormAsPW.dat};
\addlegendentry{$n=8$}
\addplot table [x=h, y=errAn9]{./GPWBAE2D2OCasecosxsiny50ptsNormAsPW.dat};
\addlegendentry{$n=9$}
\addplot table [x=h, y=errAn10]{./GPWBAE2D2OCasecosxsiny50ptsNormAsPW.dat};
\addlegendentry{$n=10$}
\addplot table [x=h, y=errAn11]{./GPWBAE2D2OCasecosxsiny50ptsNormAsPW.dat};
\addlegendentry{$n=11$}
\addplot table [x=h, y=errAn12]{./GPWBAE2D2OCasecosxsiny50ptsNormAsPW.dat};
\addlegendentry{$n=12$}
\addplot table [x=h, y=errAn13]{./GPWBAE2D2OCasecosxsiny50ptsNormAsPW.dat};
\addlegendentry{$n=13$}
\addplot table [x=h, y=errAn14]{./GPWBAE2D2OCasecosxsiny50ptsNormAsPW.dat};
\addlegendentry{$n=14$}
\addplot table [x=h, y=errAn15]{./GPWBAE2D2OCasecosxsiny50ptsNormAsPW.dat};
\addlegendentry{$n=15$}
\addplot table [x=h, y=errAn16]{./GPWBAE2D2OCasecosxsiny50ptsNormAsPW.dat};
\addlegendentry{$n=16$}
\addplot table [x=h, y=errAn17]{./GPWBAE2D2OCasecosxsiny50ptsNormAsPW.dat};
\addlegendentry{$n=17$}
\addplot table [x=h, y=errAn18]{./GPWBAE2D2OCasecosxsiny50ptsNormAsPW.dat};
\addlegendentry{$n=18$}
\addplot table [x=h, y=errAn19]{./GPWBAE2D2OCasecosxsiny50ptsNormAsPW.dat};
\addlegendentry{$n=19$}
\addplot table [x=h, y=errAn20]{./GPWBAE2D2OCasecosxsiny50ptsNormAsPW.dat};
\addlegendentry{$n=20$}
\addplot[dotted] coordinates {(2*10^-6, 10^-10) (2*10^-3, 10^-4)};
\addlegendentry{order 2}
\end{loglogaxis}
\end{tikzpicture}
\end{center}
\caption{Convergence results for the cs test case, for $n$ from $1$ to $20$, with the classical PW normalization. For each value of $n$, the expected order of convergence, namely $n+1$, is observed and the compared to Figure \ref{fig:cond} (top right) we observe a decrease of the thresholds by several orders of magnitude.}
\label{fig:csPWnorm}
\end{figure}
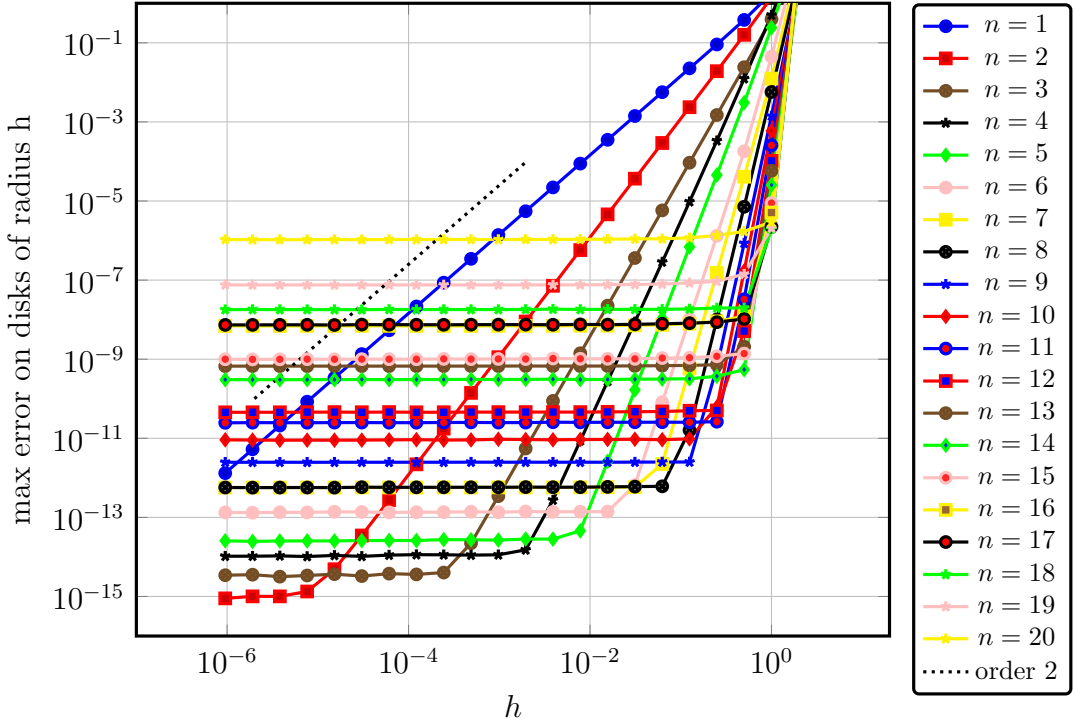

\subsection{Comparison with Phase-based GPWs}

Figure \ref{fig:JsPreAs} evidences the benefit of amplitude-based GPWs over phase-based GPWs in the pre-asymptotic regime. The construction of increasingly high order GPWs involves polynomials of increasingly high degree.
While phased-based GPWs are designed as 
$$G(x,y) := \exp P(x,y) \text{ with }\displaystyle P(x,y)=\sum_{0\leq i_x+i_y\leq \deg P} \lambda_{i_x,i_y} (x-\xG)^{i_x} (y-\yG)^{i_y},$$ 
amplitude-based GPWs are designed as 
$$G(x,y) := Q(x,y)\exp 
\begin{pmatrix} \lambda_{1,0}\\\lambda_{0,1}\end{pmatrix}\cdot \begin{pmatrix} x-\xG\\y-\yG\end{pmatrix}\text{ with }\displaystyle Q(x,y)=\sum_{0\leq i_x+i_y\leq \deg Q} \mu_{i_x,i_y} (x-\xG)^{i_x} (y-\yG)^{i_y}.$$
Therefore, in the first case, evaluating a GPW implies evaluating $ \exp \lambda_{i_x,i_y} (x-\xG)^{i_x} (y-\yG)^{i_y}$ with $0\leq i_x+i_y\leq \deg P$, while in the second case the exponential terms are limited to $\exp \lambda_{1,0}( x-\xG)$ and $\exp \lambda_{0,1}(y-\yG)$. This explains the large values of phase-based GPWs  evaluated at $(x,y)$ such that $|(x,y)-\pG|\geq 1$.
\begin{figure}
\begin{tikzpicture}
\begin{loglogaxis}[height=6.5cm,xlabel=$h$, ylabel=max error on disks of radius h,
legend pos=outer north east,
xmin=10^(-7),xmax=20,ymin=10^(-16),ymax=10,
xtick={1,0.01,.0001,.000001,.00000001},
very thick,cycle list name=mycolorlist,grid=major]
\addplot table [x=h, y=errAn1]{./GPWBAE2D2OCaseJ1xcosy50ptsCOMP.dat};
\addlegendentry{$n=1$ Ab}
\addplot table [x=h, y=errAn2]{./GPWBAE2D2OCaseJ1xcosy50ptsCOMP.dat};
\addlegendentry{$n=2$ Ab}
\addplot table [x=h, y=errAn3]{./GPWBAE2D2OCaseJ1xcosy50ptsCOMP.dat};
\addlegendentry{$n=3$ Ab}
\addplot table [x=h, y=errAn4]{./GPWBAE2D2OCaseJ1xcosy50ptsCOMP.dat};
\addlegendentry{$n=4$ Ab}
\addplot table [x=h, y=errAn5]{./GPWBAE2D2OCaseJ1xcosy50ptsCOMP.dat};
\addlegendentry{$n=5$ Ab}
\addplot table [x=h, y=errPn1]{./GPWBAE2D2OCaseJ1xcosy50ptsCOMP.dat};
\addlegendentry{$n=1$ Pb}
\addplot table [x=h, y=errPn2]{./GPWBAE2D2OCaseJ1xcosy50ptsCOMP.dat};
\addlegendentry{$n=2$ Pb}
\addplot table [x=h, y=errPn3]{./GPWBAE2D2OCaseJ1xcosy50ptsCOMP.dat};
\addlegendentry{$n=3$ Pb}
\addplot table [x=h, y=errPn4]{./GPWBAE2D2OCaseJ1xcosy50ptsCOMP.dat};
\addlegendentry{$n=4$ Pb}
\addplot table [x=h, y=errPn5]{./GPWBAE2D2OCaseJ1xcosy50ptsCOMP.dat};
\addlegendentry{$n=5$ Pb}
\addplot[dotted] coordinates {(2*10^-6, 10^-10) (2*10^-3, 10^-4)};
\addlegendentry{order 2}
\addplot[dashed] coordinates {(5*10^-2, 10^-14) (5*10^-0, 10^-2)};
\addlegendentry{order 6}
\end{loglogaxis}
\end{tikzpicture}
\begin{tikzpicture}
\begin{loglogaxis}[height=6.5cm,xlabel=$h$, 
xmin=10^(-2),xmax=20,ymin=10^(-4),ymax=10^(12),
very thick,cycle list name=mycolorlist,grid=major]
\addplot table [x=h, y=errAn1]{./GPWBAE2D2OCaseJ1xcosy50ptsCOMP.dat};
\addplot table [x=h, y=errAn2]{./GPWBAE2D2OCaseJ1xcosy50ptsCOMP.dat};
\addplot table [x=h, y=errAn3]{./GPWBAE2D2OCaseJ1xcosy50ptsCOMP.dat};
\addplot table [x=h, y=errAn4]{./GPWBAE2D2OCaseJ1xcosy50ptsCOMP.dat};
\addplot table [x=h, y=errAn5]{./GPWBAE2D2OCaseJ1xcosy50ptsCOMP.dat};
\addplot table [x=h, y=errPn1]{./GPWBAE2D2OCaseJ1xcosy50ptsCOMP.dat};
\addplot table [x=h, y=errPn2]{./GPWBAE2D2OCaseJ1xcosy50ptsCOMP.dat};
\addplot table [x=h, y=errPn3]{./GPWBAE2D2OCaseJ1xcosy50ptsCOMP.dat};
\addplot table [x=h, y=errPn4]{./GPWBAE2D2OCaseJ1xcosy50ptsCOMP.dat};
\addplot table [x=h, y=errPn5]{./GPWBAE2D2OCaseJ1xcosy50ptsCOMP.dat};
\end{loglogaxis}
\end{tikzpicture}

\begin{tikzpicture}
\begin{loglogaxis}[height=6.5cm,xlabel=$h$, ylabel=max error on disks of radius h,
legend pos=outer north east,
xmin=10^(-7),xmax=20,ymin=10^(-16),ymax=10,
xtick={1,0.01,.0001,.000001,.00000001},
very thick,cycle list name=mycolorlist,grid=major]
\addplot table [x=h, y=errAn1]{./GPWBAE2D2OCaseJ0xJ1y50ptsCOMP.dat};
\addlegendentry{$n=1$ Ab}
\addplot table [x=h, y=errAn2]{./GPWBAE2D2OCaseJ0xJ1y50ptsCOMP.dat};
\addlegendentry{$n=2$ Ab}
\addplot table [x=h, y=errAn3]{./GPWBAE2D2OCaseJ0xJ1y50ptsCOMP.dat};
\addlegendentry{$n=3$ Ab}
\addplot table [x=h, y=errAn4]{./GPWBAE2D2OCaseJ0xJ1y50ptsCOMP.dat};
\addlegendentry{$n=4$ Ab}
\addplot table [x=h, y=errAn5]{./GPWBAE2D2OCaseJ0xJ1y50ptsCOMP.dat};
\addlegendentry{$n=5$ Ab}
\addplot table [x=h, y=errPn1]{./GPWBAE2D2OCaseJ0xJ1y50ptsCOMP.dat};
\addlegendentry{$n=1$ Pb}
\addplot table [x=h, y=errPn2]{./GPWBAE2D2OCaseJ0xJ1y50ptsCOMP.dat};
\addlegendentry{$n=2$ Pb}
\addplot table [x=h, y=errPn3]{./GPWBAE2D2OCaseJ0xJ1y50ptsCOMP.dat};
\addlegendentry{$n=3$ Pb}
\addplot table [x=h, y=errPn4]{./GPWBAE2D2OCaseJ0xJ1y50ptsCOMP.dat};
\addlegendentry{$n=4$ Pb}
\addplot table [x=h, y=errPn5]{./GPWBAE2D2OCaseJ0xJ1y50ptsCOMP.dat};
\addlegendentry{$n=5$ Pb}
\addplot[dotted] coordinates {(2*10^-6, 10^-10) (2*10^-3, 10^-4)};
\addlegendentry{order 2}
\addplot[dashed] coordinates {(5*10^-3, 10^-14) (5*10^-1, 10^-2)};
\addlegendentry{order 6}
\end{loglogaxis}
\end{tikzpicture}
\begin{tikzpicture}
\begin{loglogaxis}[height=6.5cm,xlabel=$h$, 
xmin=10^(-2),xmax=20,ymin=10^(-4),ymax=10^(12),
very thick,cycle list name=mycolorlist,grid=major]
\addplot table [x=h, y=errAn1]{./GPWBAE2D2OCaseJ0xJ1y50ptsCOMP.dat};
\addplot table [x=h, y=errAn2]{./GPWBAE2D2OCaseJ0xJ1y50ptsCOMP.dat};
\addplot table [x=h, y=errAn3]{./GPWBAE2D2OCaseJ0xJ1y50ptsCOMP.dat};
\addplot table [x=h, y=errAn4]{./GPWBAE2D2OCaseJ0xJ1y50ptsCOMP.dat};
\addplot table [x=h, y=errAn5]{./GPWBAE2D2OCaseJ0xJ1y50ptsCOMP.dat};
\addplot table [x=h, y=errPn1]{./GPWBAE2D2OCaseJ0xJ1y50ptsCOMP.dat};
\addplot table [x=h, y=errPn2]{./GPWBAE2D2OCaseJ0xJ1y50ptsCOMP.dat};
\addplot table [x=h, y=errPn3]{./GPWBAE2D2OCaseJ0xJ1y50ptsCOMP.dat};
\addplot table [x=h, y=errPn4]{./GPWBAE2D2OCaseJ0xJ1y50ptsCOMP.dat};
\addplot table [x=h, y=errPn5]{./GPWBAE2D2OCaseJ0xJ1y50ptsCOMP.dat};
\end{loglogaxis}
\end{tikzpicture}
\caption{Convergence results for the JJ (bottom row) and Jc (top row) test cases, for $n$ from $1$ to $5$. Results obtained from amplitude-based and phase-based GPWs are respectively represented with solid and dashed lines for comparison. While the expected order of convergence is observed in each case, the zoom displayed on the right column evidences the advantage of the amplitude-based GPWs in the pre-asymptotic regime.}
\label{fig:JsPreAs}
\end{figure}
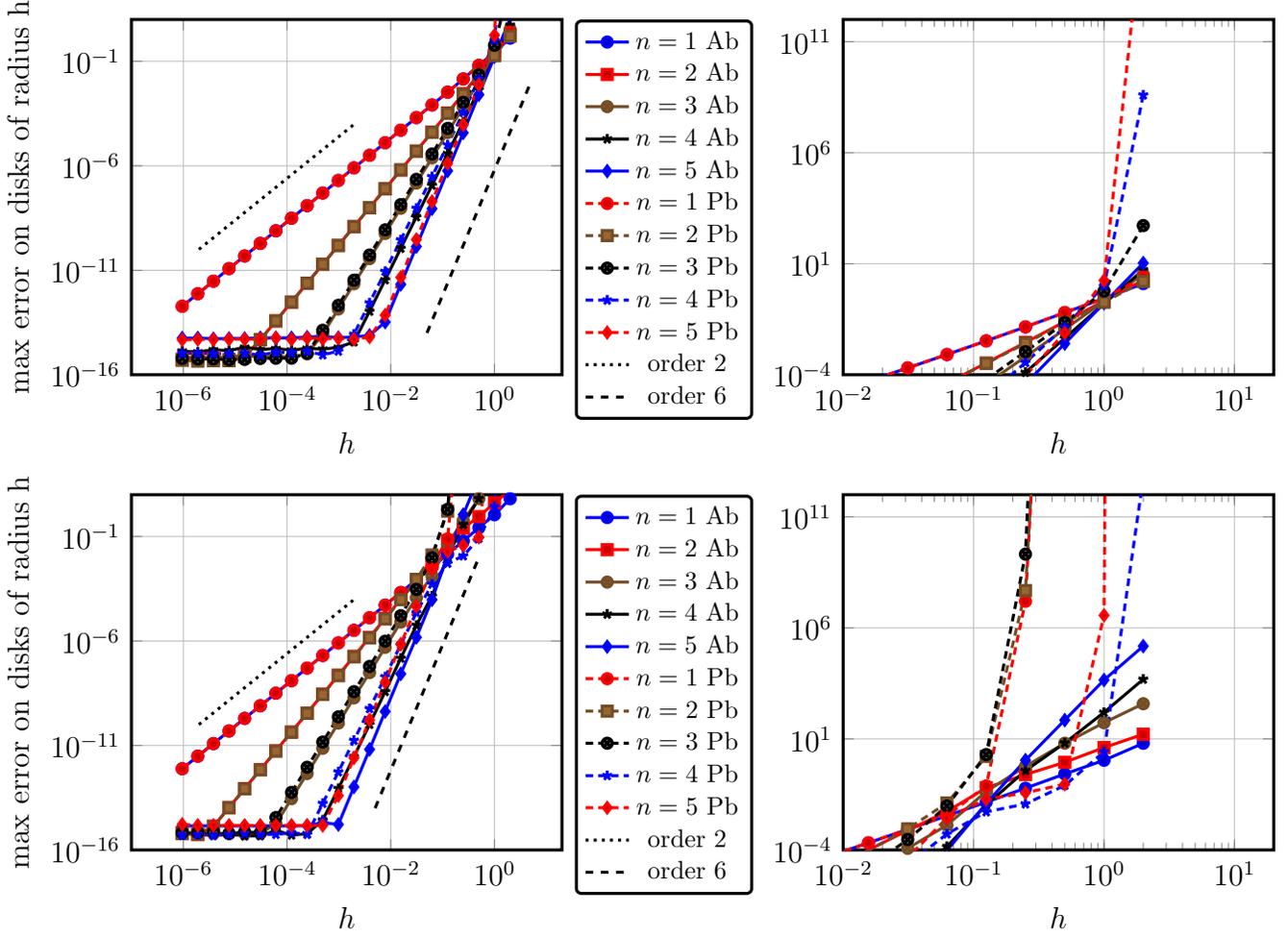

\section{Conclusion}
Amplitude-based GPWs constructed thanks to Algorithm \ref{Algo:AbGPW} are quasi-Trefftz functions, in the sense that the Taylor expansion of their image by the differential operator is zero up to a desired order.
These new GPWs enjoy the same asymptotic interpolation properties as their phase-based counterparts, and the numerical results presented in the previous section agree with the theoretical results. Moreover, we illustrate the fact that the interpolation error of the amplitude-based GPW approximation exhibits as expected a better pre-asymptotic behavior than  the interpolation error of the phase-based GPW approximation.

Future directions of research include studying the impact of the normalization on the conditioning of the GPW basis, and investigating GPWs behavior in the high frequency regime $\kappa>\!\!>1$.


\end{document}